\documentclass[letterpaper, 11pt]{amsart}

\usepackage{fullpage, amsthm, amsmath, amssymb, amsfonts}
\usepackage{graphicx}
\usepackage{subfigure}
\usepackage{hyperref}
\usepackage[enableskew]{youngtab}
\usepackage{color}
\usepackage{enumerate}

\title[Balanced labellings of affine diagrams]{Diagrams of affine permutations, balanced labellings and symmetric functions}
\author{Hwanchul Yoo and Taedong Yun}
\date{}

\newcommand{\ZZ}{\mathbb{Z}}

\newcommand{\AUU}{\widetilde{\mathcal{U}}}
\newcommand{\UU}{\mathcal{U}}

\newcommand{\R}{\mathcal{R}}

\newcommand{\B}{\mathcal{B}}

\newcommand{\CB}{\mathcal{CB}}

\newcommand{\cell}{\overline}
\newcommand{\newword}[1]{\emph{#1}}
\newcommand{\id}{\mathrm{id}}
\newcommand{\rmin}{\mathrm{rmin}}
\newcommand{\bmin}{\mathrm{bmin}}

\def\AF{\widetilde{F}}
\def\ASigma{\widetilde{\Sigma}}

\newtheorem{conj}{Conjecture}[section]

\newtheorem{theorem}[conj]{Theorem}
\newtheorem{proposition}[conj]{Proposition}

\newtheorem{lemma}[conj]{Lemma}

\newtheorem{corollary}[conj]{Corollary}

\theoremstyle{definition}
\newtheorem{definition}[conj]{Definition}

\newtheorem{remark}[conj]{Remark}
\newtheorem{example}[conj]{Example}

\newtheorem*{acknowledgement}{Acknowledgement}

\begin{document}


\begin{abstract}
We generalize the work of Fomin, Greene, Reiner, and Shimozono on balanced labellings in two directions: (1) we define the diagrams of affine permutations and the balanced labellings on them; (2) we define the set-valued version of the balanced labellings. We show that the column-strict balanced labellings on the diagram of an affine permutation yield the affine Stanley symmetric function defined by Lam, and that the column-strict set-valued balanced labellings yield the affine stable Grothendieck polynomial of Lam. Moreover, once we impose suitable flag conditions, the flagged column-strict set-valued balanced labellings on the diagram of a finite permutation give a monomial expansion of the Grothendieck polynomial of Lascoux and Sch\"{u}tzenberger. We also give a necessary and sufficient condition for a diagram to be an affine permutation diagram.
\end{abstract}

\maketitle
\tableofcontents

\section{Introduction}

The \emph{diagram}, or the \emph{Rothe diagram} of a permutation is a widely used technique to visualize the inversions of the permutation on the plane. It is well known that there is a one-to-one correspondence between the permutations and the set of their inversions.

\emph{Balanced labellings} are labellings of the diagram $D(w)$ of a permutation $w\in \Sigma_n$ such that each cell of the diagram is \emph{balanced}. They are defined in \cite{Fomin1997} to encode reduced decompositions of the permutation $w$. There is a notion of \emph{injective labellings} which generalize both standard Young tableaux and Edelman-Greene's balanced tableaux \cite{Edelman1987}, and \emph{column-strict labellings} which generalize semi-standard Young tableaux. Column-strict labellings yield symmetric functions in the same way semi-standard Young tableaux yield Schur functions. In fact, these symmetric functions $F_w(x)$ are the \emph{Stanley symmetric functions}, which was introduced to calculate the number of reduced decompositions in $\Sigma_n$ \cite{Stanley1984}. The Stanley symmetric function coincides with the Schur function when $w$ is a \emph{Grassmannian} permutation. Furthermore, if one imposes flag conditions on column strict labellings, they yield Schubert polynomial of Lascoux and Sch\"utzenberger \cite{Lascoux1985}. One can directly observe the limiting behaviour of Schubert polynomials (e.g. stability, convergence to $F_w(x)$, etc.) in this context. 

The main purpose of this paper is to extend the idea of diagrams and balanced labellings in two directions. We first define the diagrams of affine permutations and balanced labellings on them. Following the footsteps of \cite{Fomin1997}, we show that the column strict labellings on an affine permutation diagram yield the affine Stanley symmetric function defined by Lam in \cite{Lam2006}. When an affine permutation is \emph{$321$-avoiding}, the balanced labellings coincide with semi-standard cylindric tableaux, and they yield the cylindric Schur function of Postnikov \cite{Postnikov2005}.

Secondly, we define the \emph{set-valued} version of the balanced labellings of an affine permutation. In the case of $321$-avoiding finite permutations (or, equivalently, skew Young diagrams $\lambda / \mu$), our definition of the set-valued balanced labellings coincides with the \emph{set-valued tableaux} appearing in \cite{Buch2002}. We show that the column-strict set-valued balanced labellings on an affine permutation diagram give a monomial expansion of the affine stable Grothendieck polynomial of Lam \cite{Lam2006}. Moreover, in the case of finite permutations, we impose flag conditions to the column-strict set-valued balanced labellings to get an expansion of the Grothendieck polynomial of Lascoux and Sch\"{u}tzenberger \cite{Lascoux1983}.


An interesting byproduct of balanced labellings is a complete characterization of diagrams of affine permutations using the notions of a \emph{content}. We will introduce the notion of a \emph{wiring diagram} of an affine permutation diagram in the process, which generalizes Postnikov's wiring diagram of Grassmannian permutations \cite{Postnikov2006}.

\begin{remark}
This paper contains the full version of the extended abstract submitted to FPSAC'13.
\end{remark}


\section{Permutation Diagrams and Balanced Labellings}

\subsection{Permutations and Affine Permutations}~\label{sec:permutation}

Let $\Sigma_n$ denote the \emph{symmetric group}, the group of all permutations of size $n$. $\Sigma_n$ is generated by the \emph{simple reflections} $s_1, \ldots, s_{n-1}$, where $s_i$ is the permutation which interchanges the entries $i$ and $i+1$, and the following relations.
\begin{eqnarray*}
&s_i^2=1 &\text{for all } i \\
&s_i s_{i+1} s_i = s_{i+1} s_i s_{i+1} &\text{for all } i \\
&s_i s_j = s_j s_i &\text{for }|i-j|\ge 2
\end{eqnarray*}
In this paper, we will often call a permutation a \emph{finite} permutation and the symmetric group the \emph{finite} symmetric group to distinguish them from its affine counterpart.

On the other hand, the \emph{affine symmetric group} $\ASigma_n$ is the group of all \emph{affine permutations} of period $n$. A bijection $w: \ZZ \rightarrow \ZZ$ is called an affine permutation of period $n$ if $w(i + n) = w(i) + n$ and $\sum_{i = 1}^n w(i) = n(n+1)/2$. An affine permuation is uniquely determined by its \emph{window}, $[w(1), \ldots, w(n)]$, and by abuse of notation we write $w = [w(1), \ldots, w(n)]$ (window notation).

We can describe the group $\ASigma_n$ by its generators and relations as we did with $\Sigma_n$. The generators are $s_0, s_1, \ldots, s_{n-1}$ where $s_i$ interchanges all the \emph{periodic} pairs $\{ (kn + i, kn + i + 1) \mid k \in \ZZ \}$. With these generators we have exactly the same relations
\begin{eqnarray*}
&s_i^2=1 &\text{for all } i \\
&s_i s_{i+1} s_i = s_{i+1} s_i s_{i+1} &\text{for all } i \\
&s_i s_j = s_j s_i &\text{for }|i-j|\ge 2
\end{eqnarray*}
but here all the indices are taken modulo $n$, i.e. $s_{n+i} = s_{i}$. Note that the symmetric group can be embedded into the affine symmetric group by sending $s_i$ to $s_i$. With this embedding, we will identify a finite permutation $w = [w_1, \ldots, w_n]$ with the affine permutation $[w_1, \ldots, w_n]$ written in the window notation. 

A \emph{reduced decomposition} of $w$ is a decomposition $w=s_{i_1}\cdots s_{i_\ell}$ where $\ell$ is the minimal number for which such a decomposition exists. In this case, $\ell$ is called the \emph{length} of $w$ and denoted by $\ell(w)$. The word $i_1i_2\cdots i_\ell$ is called a \emph{reduced word} of $w$. It is well-known that the length of an affine permutation $w$ is the same as the cardinality of the set of \emph{inversions}, $\{(i,j) \mid 1 \leq i \leq n,~ i<j,~ w(i)>w(j)\}$.

The \emph{diagram}, or \emph{affine permutation diagram}, of $w\in\ASigma_n$ is the set 
$$D(w)=\{(i,w(j)) \mid i<j, w(i)>w(j)\}\subseteq \ZZ\times \ZZ.$$
This is a natural generalization of the \emph{Rothe diagram} for finite permutations. When $w$ is finite, $D(w)$ consists of infinite number of identical copies of the Rothe diagram of $w$ diagonally. From the construction it is clear that $(i,j)\in D(w) \Leftrightarrow (i+n,j+n)\in D(w)$.

Throughout this paper, we will use a matrix-like coordinate system on $\ZZ\times\ZZ$: The vertical axis corresponds to the first coordinate increasing as one moves toward south, and the horizontal axis corresponds to the second coordinate increasing as one moves toward east. We will visualize $D(w)$ as a collection of unit square lattice boxes on $\ZZ\times\ZZ$ whose coordinates are given by $D(w)$.

\begin{figure}
\centering
\includegraphics{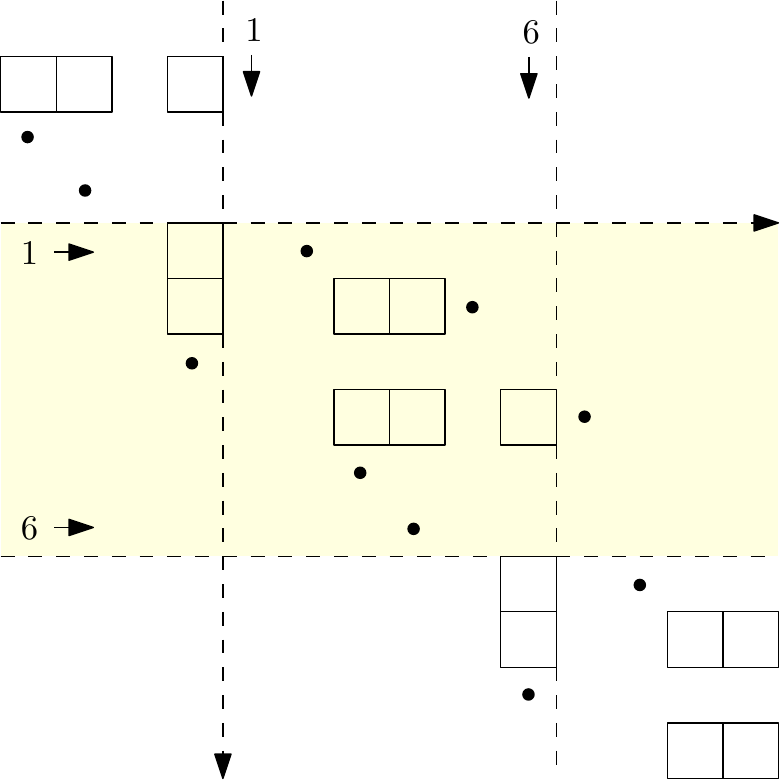}
\caption{diagram of $[2,5,0,7,3,4] \in \ASigma_6$}
\label{fig:diag}
\end{figure}

\bigskip

\subsection{Diagrams and Balanced Labellings}~\label{sec:diagram_balanced}

We call a collection $D$ of unit square lattice boxes on $\ZZ\times\ZZ$ an \emph{affine diagram }(\emph{of period} $n$) if there are finite number of cells on each row and column, and $(i,j)\in D \Leftrightarrow (i+n,j+n)\in D$. Obviously $D(w)$ is an affine diagram of period $n$. In an affine diagram, the collection of boxes $\{(i+rn,j+rn)\mid r\in\ZZ\}$ are called the \emph{cell} of $D$, and we will denote it by $\cell{(i,j)}$. From the periodicity, we can take the representative of each cell $\cell{(i,j)}$ in the first $n$ rows $\{1,2,\ldots,n\}\times \ZZ$, called the \emph{fundamental window}. Each horizontal strap $\{1+rn,\cdots,n+rn\}\times\ZZ$ for some $r\in\ZZ$ will be called a \emph{window}. The intersection of $D$ and the fundamental window will be denoted by $[D]$. The boxes in $[D]$ are the natural representatives of the cells of $D$. An affine diagram $D$ is said to be of the size $\ell$ if the number of boxes in $[D]$ is $\ell$. Note that the size of $D(w)$ for $w\in\ASigma_n$ is the length of $w$.

\begin{example}
The length of an affine permutation $w = [2,5,0,7,3,4] \in \ASigma_6$ in Figure~\ref{fig:diag} is $7$, e.g. $w = s_0 s_4 s_5 s_3 s_4 s_1 s_2$, and hence its fundamental window (shaded region) contains $7$ boxes. The dots represent the permutation and the square boxes represent the diagram in the figure.
\end{example}

To each cell $(i,j)$ of an affine diagram $D$, we associate the \emph{hook} $H_{i,j}:=H_{i,j}(D)$ consisting of the cells $(i',j')$ of $D$ such that either $i'=i$ and $j'\ge j$ or $i'\ge i$ and $j'=j$. The cell $(i,j)$ is called the \emph{corner} of $H_{i,j}$.

\begin{definition}[Balanced hooks]~\label{def:bal_hook}
A labellings of the cells of $H_{i,j}$ with positive integers is called \emph{balanced} if it satisfies the following condition: if one rearranges the labels in the hook so that they weakly increase from right to left and from top to bottom, then the corner label remains unchanged.
\end{definition}

\begin{figure}
\centering
\includegraphics{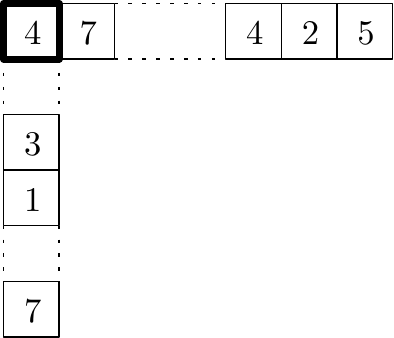}
\caption{balanced hook}
\label{fig:bhook}
\end{figure}

A labelling of an affine diagram is a map $T:D\rightarrow \ZZ_{>0}$ from the boxes of $D$ to the positive integers such that $T(i,j)=T(i+n,j+n)$ for all $(i,j)\in D$. In other words, it sends each cell $\cell{(i,j)}$ to some positive integer. Therefore if $D$ has size $\ell$, there can be at most $\ell$ different numbers for the labels of the boxes in $D$.

\begin{definition}[Balanced labellings]
Let $D$ be an affine diagram of the size $\ell$.
\begin{enumerate}
\item A labelling of $D$ is \emph{balanced} if each hook $H_{i,j}$ is balanced for all $(i,j)\in D$.
\item A balanced labelling is \emph{injective} if each of the labels $1,\cdots,\ell$ appears exactly once in $[D]$.
\item A balanced labelling is \emph{column strict} if no column contains two equal labels.
\end{enumerate}
\end{definition}

\begin{figure}
\centering
\includegraphics{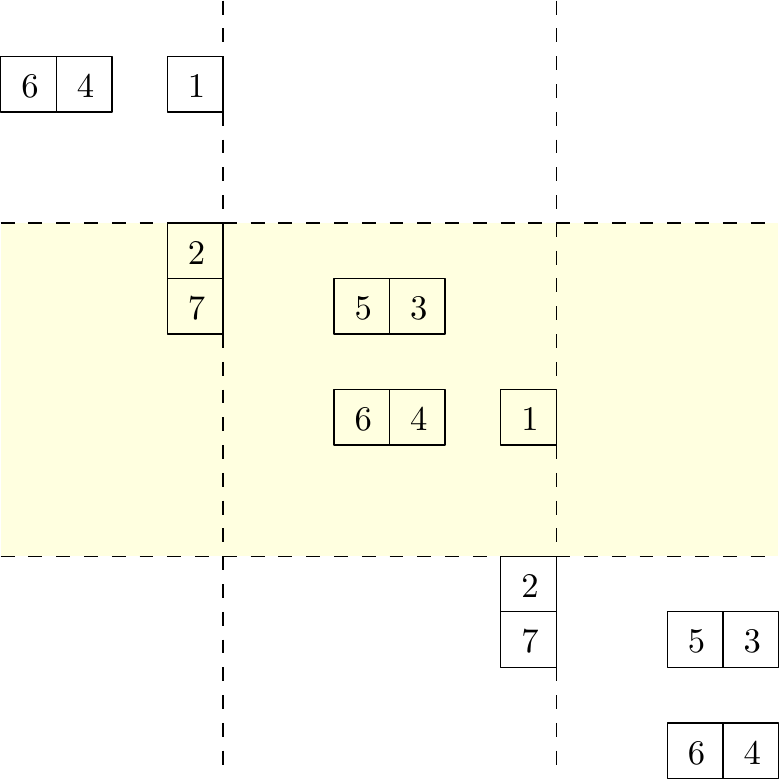}
\caption{injective balanced labelling}
\label{fig:blabel}
\end{figure}

\bigskip

\subsection{Injective Labellings and Reduced Words}~\label{sec:injective} 

Given $w\in\ASigma_n$ and its reduced decomposition $w=s_{a_1}\cdots s_{a_\ell}$, 
we read from left to right and interpret $s_k$ as adjacent transpositions switching the numbers at $(k+rn)$-th and $(k+1+rn)$-th positions, for all $r\in\ZZ$. In other words, $w$ can be obtained from applying the sequence of transpositions $s_{a_1}, s_{a_2}, \ldots, s_{a_\ell}$ to the identity permutation. It is clear that each $s_i$ corresponds to a unique inversion of $w$. Here, an \emph{inversion} of $w$ is a family of pairs $\{(w(i+rn),w(j+rn))\mid r\in\ZZ\}$ where $i<j \text{ and } w(i)>w(j)$. Note that $w(i+rn)> w(j+rn) ~\Leftrightarrow~ w(i) > w(j)$. Often we will ignore $r$ and use a representative of pairs when we talk about the inversions. On the other hand, each cell of $D(w)$ also corresponds to a unique inversion of $w$. In fact, $(i,j)\in D(w)$ if and only if $(w(i),j)$ is an inversion of $w$. 

\begin{definition}[Canonical labelling]
Let $w\in\ASigma_n$ be of length $\ell$, and $a = a_1 a_2 \cdots a_\ell$ be a reduced word of $w$. Let $T_a:D\rightarrow \{1,\cdots,\ell\}$ be the injective labelling defined by setting $T_a(i,w(j))=k$ if $s_{a_k}$ transposes $w(i)$ and $w(j)$ in the partial product $s_{a_1}\cdots s_{a_k}$ where $w(i)>w(j)$. Then $T_a$ is called the \emph{canonical labelling} of $D(w)$ induced by $a$.
\end{definition}

\begin{figure}
\centering
\includegraphics{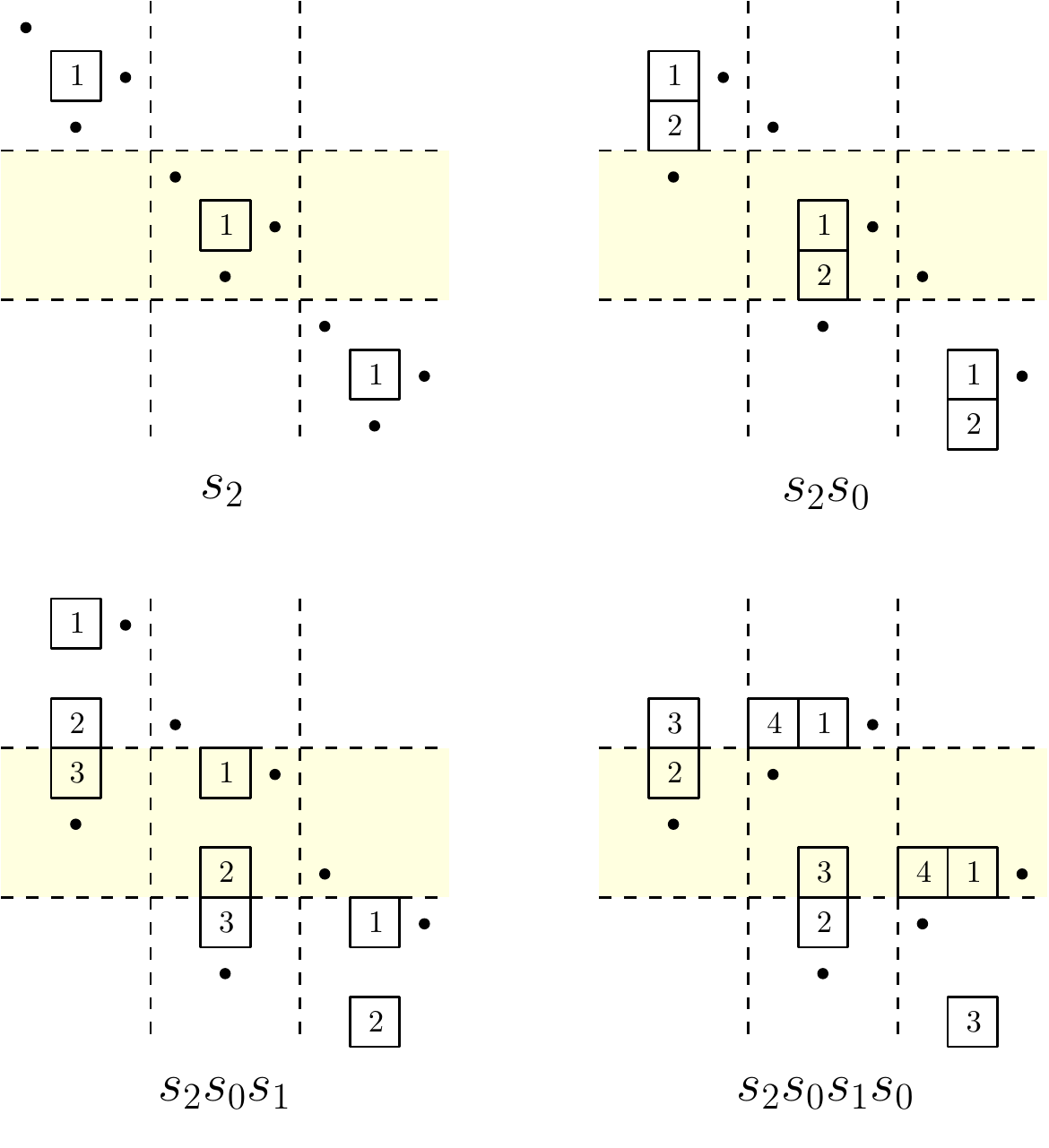}
\caption{canonical labelling of $s_2 s_0 s_1 s_0 \in \ASigma_3$}
\label{fig:canonical}
\end{figure}

\begin{proposition}~\label{prop:canonical_balanced}
A canonical labelling of a reduced word of an affine permutation $w$ is an injective balanced labelling.
\end{proposition}

Before we give a proof of Proposition~\ref{prop:canonical_balanced}, we introduce our main tool for proving that a given labelling is balanced. The following lemma is closely related to the notion of \emph{normal ordering} in a root system. 

\begin{lemma}[Localization]\label{lem:localization}
Let $w\in\ASigma_n$ and let $T$ be a column strict labelling of $D(w)$. Then $T$ is balanced if and only if for all integers $i < j < k$ the restriction of $T$ to the sub-diagram
of $D(w)$ determined by the intersections of rows $i,j,k$ and columns $w(i),w(j),w(k)$ is balanced.
\end{lemma}

\begin{proof}
($\Longleftarrow$) Given a labelling $T$ of a diagram of an affine permutation $w$, suppose that the labelling is balanced for all subdiagrams $D_{ijk}$ determined by rows $\{i, j, k\}$ and columns $\{w(i), w(j), w(k)\}$. Let $(i, w(j))$ be an arbitrary box in the diagram such that $i < j$ and $ w(i) > w(j)$ and let $a = T(i, w(j))$. By abuse of notation, we will denote by $a$ the box itself. Let us call all the boxes to the right of $a$ in the same row the \emph{right-arm} of $a$, and all the boxes below $a$ in the same column the \emph{bottom-arm} of $a$. To show that the diagram is balanced at $a$, we need to show that there is a injection $\phi_a$ from the set $B^{<}_a$ of all boxes in the bottom-arm of $a$ whose labelling is less than $a$, into the set $R^{\geq}_a$ of all boxes in the right-arm of $a$ whose labelling is greater than or equal to $a$, such that the image of $\phi_a$ contains the set $R^{>}_a$ of all boxes in the right-arm of $a$ whose labelling is \emph{greater} than $a$. Let $(p, w(j))$ be a box in the bottom hook of $a$ such that $T(i, p) < a$. By the balancedness of the $D_{i, p, j}$, $w(i) > w(p) > w(j)$ and $T(i, w(p)) \geq a$. Let $\phi_q$ be the map defined by $(p, w(j)) \mapsto (i, w(p))$. It is easy to see that every box on the right-arm of $a$ whose labelling is greater than $a$ should be an image of $\phi_a$ by a similar argument so $\phi_a$ is the desired injection.

($\Longrightarrow$) Suppose a labelling $T$ of a diagram of an affine permutation $w$ is balanced. Since the diagram is balanced at any point $x$, there is a bijection $\phi_x$ from $B^{<}_x$ to a subset $M$ of the boxes in the right-arm of $x$ such that $R^{>}_x \subset M \subset R^{\geq}_x$. For an element $y$ in $M$, we will write $\phi_x(y)$ instead of $\phi_x^{-1}(y)$ for simplicity.

The nine points in $D_{ijk}$ ($i < j < k$) may contain 0, 1, 2, or 3 boxes (since the maximum number of  inversions of size $3$ permutations is $3$.) Let $p < q < r$ be the rearrangement of $w(i), w(j), w(k)$. The labelling of the boxes of the intersection is clearly balanced when it has $0$ or $1$ boxes, or when it has $2$ boxes and two labellings are the same. Therefore we only need to consider the following three cases.

\begin{enumerate}
\item[Case 1.] Two boxes at $(i, p)$ and $(i, q)$ (i.e. $w(j) < w(k) < w(i)$).

$a = T(i, p), b = T(i, q)$. To show $a \geq b$, we use induction on $j - i$. When $j - i = 1$, the balancedness at $a$ directly implies $a \geq b$.

Suppose $a < b$ for contradiction. Let $c = \phi_a(b)$ be the box in the bottom-arm of $a$ which corresponds to $b$ via $\phi_a$ (thus $c < a$), and let $\ell$ be the row index of the box $c$. Here we have two cases.

\begin{enumerate}[(1)]
\item $p < w(\ell) < q$\\
Let $e$ be the box at the intersection of the right-arm of $a$ and the column $w(\ell)$. Applying induction hypothesis to $D_{i, \ell, k}$, we get $e \geq b (> a)$. Hence, we may apply $\phi_a$ to $e$ and let $c_1 = \phi_a(e)$.

\item $q < w(\ell)$\\
Let $d$ be the box at $(\ell, q)$. By induction hypothesis to $D_{\ell, j, k}$, we get $d \leq c (< a < b)$. Since $d < b$, let $e = \phi_b(d)$. Here, $e \geq b > a$ so let $c_1 = \phi_a(e)$.
\end{enumerate}

In both case we get a box $c_1$ in the bottom-arm of $a$, which is less than $a$ and distinct from $c$. We may repeat the same process with $c_1$ as we did with $c$, and compute another point $c_2$ in the bottom-arm of $a$, which is less than $a$ and distinct from $c$ and $c_1$, and we can continue this process. The construction of $c_i$ ensures that $c_i$ is distinct from any of $c, c_1, \ldots, c_{i - 1}$. This is a contradiction since there are finite number of boxes in the bottom-arm of $a$.

\item[Case 2.] Two boxes at $(i, p)$ and $(j, p)$ (i.e. $w(k) < w(i) < w(j)$).

The symmetric version of the proof of Case 1 will work here if we switch rows with columns and reverse all the inequalities.

\item[Case 3.] Three boxes at $(i, p)$, $(i, q)$, and $(j, p)$ (i.e. $w(k) < w(j) < w(i)$).

We use induction on $\min \{k - i, r - p\}$. Let $T(i, p) = a$, $T(i, q) = b$, and $T(j, p) = c$.

For the base case where $\min \{k - i, r - p\} = 2$, we may assume $r - p = 2$ by symmetry. Note that $q = p + 1$ and $r = q + 1$. If $a$ is not balanced in the $D_{i, j, k}$, then both $b$ and $c$ should be greater than $a$. (If both $b$ and $c$ are smaller than $a$, than the hook at $a$ cannot be balanced.) This implies that there is a box $\phi_a(b) = d$ on the bottom-arm of $a$ such that $d < a$. If $d$ is above $c$, then $a$ and $d$ contradicts the result in Case 2. If $d$ is below $c$, then $c$ is not balanced in the diagram, which contradicts the assumption. This completes the proof of the base case. 

Now, let $a$ be smaller than both $b$ and $c$. As before, there is a box $\phi_a(b) = d_1$ on the bottom-arm of $a$ such that $d_1 < a$. Let the row index of $d_1$ be $\ell$.

\begin{enumerate}[(1)]
\item $\ell<j$ and $w(\ell)<q$.\\
Let $e$ be the label of the box $(i,w(\ell))$. By applying the result of Case 1 to $e$ and $b$, we get $e \geq b (> a)$. Thus there must be another $d_2=\phi_a(e)$ in the bottom-arm of $a$ such that $a>d_2$.

\item $\ell<j$ and $q<w(\ell)<r$.\\
Let $e$ be the label of the box $(i,w(\ell))$, and $f$ be the label of the box $(\ell,q)$. By the induction hypothesis, $D_{\ell, j, k}$ is balanced, so $d_1 \geq f$. This implies $f < b$, and by the induction hypothesis, $D_{i, \ell, j}$ also form a balanced subdiagram. Hence $e \geq b (> a)$. Therefore we have another box $\phi_a(e) = d_2$ such that $a > d_2$.

\item $\ell<j$ and $r<w(\ell)$.\\
This is impossible because $a < d_1$ by Case 2, which contradicts our choice of $d_1$.

\item $\ell>j$ and $w(\ell)<q$.\\
Let $e = T(i, w(\ell))$, and $f = T(j, w(\ell))$. By the induction hypothesis, $f, c, d_1$ form a balanced subdiagram, so $c \leq f$. Similarly, $b, e, f$ form a balanced subdiagram. Since $b$ and $f$ are both greater than $a$, so is $e$. Therefore we have $\phi_a(e) = d_2 < a$ on the bottom-arm of $a$.

\item $\ell>j$ and $w(\ell)>q$.\\
This case is impossible because $c>d$ by Case 2 which is a contradiction.

\end{enumerate}

After we get $d_2$ in the above, we can repeat the argument for $d_2$ instead of $d_1$. The construction of $d_i$ ensures that $d_i$ is distinct from any of $d_1, \ldots, d_{i - 1}$. This is a contradiction since there are finite number of boxes in the bottom-arm of $a$. 

When $a$ is greater than both $b$ and $c$, the transposed version of the above argument works by symmetry. So we are done.
\end{enumerate}
\end{proof}

Now we are ready to prove our proposition.

\begin{proof}[Proof of Proposition~\ref{prop:canonical_balanced}]
A canonical labelling is injective by its construction. By Lemma~\ref{lem:localization}, it is enough to show that for any triple $i < j < k$ the intersection $D_{ijk}$ of the canonical labelling of $D(w)$ with the rows $i, j, k$ and the columns $w(i), w(j), w(k)$ is balanced.

Let $p < q < r$ be the rearrangement of $w(i), w(j), w(k)$. As we have seen in the proof of Lemma~\ref{lem:localization}, $I$ is clearly balanced when $I$ contains $0$ or $1$ boxes, hence we only need to consider the following three cases.

\begin{enumerate}[(1)]
\item $w(j) < w(k) < w(i)$, two horizontal boxes in $D_{ijk}$\\
In this case $w = [\ldots, r, \ldots, p, \ldots, q, \ldots]$ if one write down the affine permutation. When we apply simple reflections in a reduced word of $w$ one-by-one from left to right, to get $w$ from the identity permutation $[\ldots, p, \ldots, q, \ldots, r, \ldots]$, $r$ should pass through $q$ before it passes through $p$ (because the relative order of $p$ and $q$ should stay the same throughout the process). This implies that the canonical labelling of the right box is less than the canonical labelling of the left box, and hence $D_{ijk}$ is balanced.

\item $w(k) < w(i) < w(j)$, two vertical boxes in $D_{ijk}$\\
In this case $w = [\ldots, q, \ldots, r, \ldots, p, \ldots]$. By a similar argument $p$ should pass through $q$ before it passes through $r$ when we apply simple reflections. This implies the canonical labelling of the bottom box is greater than the canonical labelling of the top box.

\item $w(k) < w(j) < w(i)$, three boxes in $D_{ijk}$ in ``$\Gamma$''-shape\\
$w = [\ldots, r, \ldots, q, \ldots, p, \ldots]$ in this case. If $p$ passes through $q$ before $r$ passes through $q$, then $r$ should pass through $p$ before it passes through $q$. This implies that the canonical labelling of the corner box lies between the labellings of other two boxes. If $r$ passes through $q$ before $p$ passes through $q$, then again by a similar argument the corner box lies between the labelling of other two boxes. Hence, $D_{ijk}$ is balanced.
\end{enumerate}

We have showed that$D_{ijk}$ is balanced for every triple $i,j,k$ and thus by Lemma~\ref{lem:localization} the canonical labelling of $D(w)$ is balanced.
\end{proof}

Conversely, suppose we are given an injective labelling of an affine permutation diagram $D(w)$. Is every injective labelling a canonical labelling of a reduced word? To answer this question we first introduce some terminology.

\begin{definition}[Border cell]
Let $w\in\ASigma_n$ and $\cell{(i,j)}$ be a cell of $D(w)$. If $w(i+1)=j$ then the cell $\cell{(i,j)}$ is called a \emph{border cell} of $D(w)$.
\end{definition}

The border cells correspond to the (right) descents of $w$, i.e. the simple reflections that can appear at the end of some reduced decomposition of $w$. When we multiply a descent of $w$ to $w$ from the right, we get an affine permutations whose length is $\ell(w)-1$. It is easy to see that this operation changes the diagram in the following manner.

\begin{lemma}\label{lem:removing_border_cell}
Let $s_i$ be a descent of $w$, and $\alpha = \cell{(i,j)}$ be the corresponding border cell of $D(w)$. Let $D(w)\setminus \alpha$ denote the diagram obtained from $D(w)$ by deleting every boxes $(i+rn,j+rn)$ and exchanging rows $(i+rn)$ and $(i+1+rn)$, for all $r\in\ZZ$. Then the diagram $D(ws_i)$ is $D(w)\setminus \alpha$. \qed
\end{lemma}

\begin{lemma}\label{lem:max_label}
Let $T$ be a column strict balanced labelling of $D(w)$ with largest label $M$, then every row containing an $M$ must contain an $M$ in a border cell. In particular, if $i$ is the index of such row, then $i$ must be a descent of $w$.
\end{lemma}
\begin{proof}
Suppose that the row $i$ contains an $M$. First we show that $i$ is a descent of $w$. If $i$ is not a descent, i.e. $w(i) < w(i+1)$, then let $(i,j)$ be the rightmost box in row $i$ whose labelling is $M$. Since $w(i) < w(i+1)$, there is a box at $(i+1, j)$. By column-strictness no box below $(i,j)$ has label $M$ and no box to the right of $(i,j)$ has label $M$ by the assumption. Hence the diagram is not balanced at $(i,j)$, which is a contradiction. Therefore $i$ must be a descent of $w$.

Let $w(i+1) = j$, i.e. $(i,j)$ is a border cell. We must show that $T(i,j) = M$. If $T(i,j) < M$, then the rightmost occurrence of $M$ cannot be to the right of $(i,j)$ because the hook $H_{ij}$ is horizontal. On the other hand, if the rightmost occurrence of $M$ is to the left of $(i,j)$, then there must be a box below that rightmost $M$ and the hook at that $M$ is not balanced by the argument in the previous paragraph. Hence, $T(i,j) = M$.
\end{proof}

\begin{theorem}\label{thm:removing_border_balanced}
Let $T$ be a column strict labelling of $D(w)$, and assume some border cell $\alpha$ contains the largest label $M$ in $T$. Let $T\setminus \alpha$ be the result of deleting all the boxes of $\alpha$ and switching pairs of rows $(i + rn, i + 1 + rn)$ for all $r\in \ZZ$ from $T$. Then $T$ is balanced if and only if $T \setminus \alpha$ is balanced.
\end{theorem}
\begin{proof}
Let $\alpha = \cell{(i,j)}$ be the border cell, and $w' = w s_i$ so that $T \setminus \alpha$ is a labelling of $D(w')$. By Lemma~\ref{lem:localization}, it suffices to show that for all $a < b < c$ the restriction $T_{abc}$ of $T$ to the subdiagram of $D(w)$ determined by rows $a,b,c$ and columns $w(a), w(b), w(c)$ is balanced if and only if the restriction $(T \setminus \alpha)_{s_i a, s_i b, s_i c}$ is balanced.

Note that for every $(r, s)$ the $(r, w(s))$-entry of $T$ coincides with the $(s_i r, w(s))$-entry of $T\setminus \alpha$ unless $(r, w_s) = (i + rn, j + rn)$ for some $r\in \ZZ$. Hence $T_{abc}$ will be the same as $(T\setminus \alpha)_{s_i a, s_i b, s_i c}$ unless $i + rn \in \{a, b, c\}$ and $j + rn \in \{(w(a), w(b), w(c)\}$ for some $r\in \ZZ$. Therefore we may assume we are in this case, so $(T\setminus \alpha)_{s_i a, s_i b, s_i c}$ has one fewer box than $T_{abc}$. Furthermore, if $T_{abc}$ has at most two boxes (and $(T\setminus \alpha)_{s_i a, s_i b, s_i c}$ has at most one box), then the verification is trivial since $M$ is the largest label and $(i,j)$ is a border cell.

Thus we may assume that $T_{abc}$ has three boxes and $(T\setminus \alpha)_{s_i a, s_i b, s_i c}$ has two boxes, so $w(c) < w(b) < w(a)$ and either $(a, b) = (i + rn, i + 1 + rn)$ or $(b, c) = (i + rn, i + 1 + rn)$ for some $r \in \ZZ$. In the first case $T_{abc}$ being balanced and $(T\setminus \alpha)_{abc}$ being balanced are both equivalent to the condition $T(a, w(c)) \geq T(b, w(c))$, and in the second case they are both equivalent to the condition $T(a, w(c)) \geq T(a, w(b))$.
\end{proof}

Combining Proposition~\ref{prop:canonical_balanced}, Lemma~\ref{lem:max_label} and Theorem~\ref{thm:removing_border_balanced}, we obtain the main theorem of this section.

\begin{theorem}~\label{thm:canonical_labelling}
Let $\R(w)$ denote the set of reduced words of $w\in\ASigma_n$, and $\B(D)$ denote the set of injective balanced labellings of the affine diagram $D$. The correspondence $a\mapsto T_a$ is a bijection between $\R(w)$ and $\B(D(w))$. \qed
\end{theorem}


An algorithm to decode the reduced word from a balanced labelling will be given in Section~\ref{sec:encoding_decoding}. Another immediate corollary of Theorem~\ref{thm:removing_border_balanced} is a recurrence relation on the number of injective balanced labellings.

\begin{corollary}\label{cor:recurrence_injective}
Let $b_{D(w)}$ denote the number of injective balanced labellings of $D(w)$. Then,
$$b_{D(w)} = \sum_{\alpha} b_{D(w)\setminus \alpha},$$
where the sum is over all border cells $\alpha$ of $D(w)$. \qed
\end{corollary}

\subsection{Column Strict Tableaux and Affine Stanley Symmetric Functions}~\label{sec:column_strict}

In this section we consider column strict balanced labellings of affine permutation diagrams. We show that they give us the affine Stanley symmetric function in the same way the semi-standard Young tableaux give us the Schur function.

Affine Stanley symmetric functions are symmetric functions parametrized by affine permutations. They are defined in \cite{Lam2006} as an affine counterpart of the Stanley symmetric function \cite{Stanley1984}. Like Stanley symmetric functions, they play an important role in combinatorics of reduced words. The affine Stanley symmetric functions also have natural geometric interpretation \cite{Lam2008}, namely they are pullbacks of the cohomology Schubert classes of the affine flag variety $LSU(n)/T$ to the affine Grassmannian $\Omega SU(n)$ under the natural map $\Omega SU(n) \rightarrow LSU(n)/T$. There are various ways to define the affine Stanley symmetric function, including the geometric one above. For our purpose, we use one of the two combinatorial definitions in \cite{Lam2010}.

A word $a_1a_2\cdots a_\ell$ with letters in $\ZZ/n\ZZ$ is called \emph{cyclically decreasing} if (1) each letter appears at most once, and (2) whenever $i$ and $i+1$ both appears in the word, $i+1$ precedes $i$. An affine permutation $w\in\ASigma_n$ is called cyclically decreasing if it has a cyclically decreasing reduced word. We call $w=v_1v_2\cdots v_r$ \emph{cyclically decreasing factorization} of $w$ if each $v_i\in\ASigma_n$ is cyclically decreasing, and $\ell(w)=\sum_{i=1}^{r} \ell(v_i)$. We call $(\ell(v_1), \ell(v_2), \ldots, \ell(v_r))$ the \emph{type} of the cyclically decreasing factorization.

\begin{definition}[\cite{Lam2010}]
Let $w\in\ASigma_n$ be an affine permutation. The affine Stanley symmetric function $\AF_w(x)$ corresponding to $w$ is defined by
$$\AF_w(x):=\AF_w(x_1,x_2,\cdots)=\sum_{w=v_1v_2\cdots v_r}x_1^{\ell(v_1)}x_2^{\ell(v_2)}\cdots x_r^{\ell(v_r)},$$
where the sum is over all cyclically decreasing factorization of $w$.
\end{definition}

Given an affine diagram $D$, let $\CB(D)$ denote the set of column strict balanced labellings of $D$. Now we can state the main theorem of this section.

\begin{theorem}\label{thm:affine_stanley}
Let $w\in\ASigma_n$ be an affine permutation. Then
$$\AF_w(x)=\sum_{T\in \CB(D(w))}x^T$$
where $x^T$ denotes the monomial $\prod_{(i,j)\in [D(w)]}x_{T(i,j)}$
\end{theorem}
\begin{proof}
Given a column strict balanced labelling $T$, we call the sequence $([$the number of $1$'s in $T]$, $[$the number of $2$'s in $T]$, $\ldots )$ the \emph{type} of the labelling. It is enough to show that there is a type-preserving bijection $\phi$ from a column strict labelling of $D(w)$ to a cyclically decreasing factorization of $w$.

Let us construct $\phi$ as follows. Given a column strict labelling $T$ with $t$ cells, there is a (not necessarily unique) border cell $c_1$ which contains the largest label of $T$ by Lemma~\ref{lem:max_label}. Let $r(c_1)$ be the row index of $c_1$ in the fundamental window. By Theorem~\ref{thm:removing_border_balanced}, we obtain a column strict balanced labelling $T\setminus c_1$ by removing the cell $c_1$ and switching all pairs of rows $(r(c_1) + kn, r(c_1) + kn + 1)$ for all $k \in \ZZ$. The diagram of this labelling corresponds to the affine permutation $w s_{r(c_1)}$ with length $t - 1$. In $T\setminus c_1$ we again pick a border cell $c_2$ containing the largest label of $T \setminus c_1$ and remove the cell to get a labelling $T \setminus c_1 \setminus c_2$ of $w s_{r(c_1)} s_{r(c_2)}$. We continue this process removing cells $c_1, c_2, \ldots, c_t$ until we get the empty digram which corresponds to the identity permutation. Then, $w = s_{r(c_t)} s_{r(c_{t-1})} \cdots, s_{r(c_1)}$ is a reduced decomposition of $w$. Now in this reduced decomposition, group the terms together in the parenthesis if they correspond to removing the same largest label of the digram in the process and this will give you a factorization $\phi(T)$ of $w$. We will show that this is indeed a cyclically decreasing factorization and that this map is well-defined.

We first show that the words inside each parenthesis is cyclically decreasing. If the index $i = r(c_x)$ and $i + 1 = r(c_y)$ is in the same parenthesis in $\phi(T)$, then they corresponds to removing the border cells of the same largest labelling $M$ in the above process. We want to show that $i + 1$ precedes $i$ inside the parenthesis. If $i$ precede $i + 1$ in the parenthesis, then it implies we unwind the descent at $i + 1$ before we unwind the descent at $i$ during the process. Then at the time when we removed the border cell $c_y$ at $i+1$-st row with label $M$, the cell right above $c_y$ was $c_x$ with label $M$. This contradicts the column-strictness of the diagram so $i + 1$ should always precede $i$ if they are inside the same parenthesis.

Now we show that $\phi$ is well-defined. It is enough to show that if we had two border cells $c_x$ and $c_y$ with the same largest labelling at some point (so we had a choice of taking one before another) then $\lvert r(c_x) - r(c_y) \rvert \geq 2$ so the corresponding simple reflections commute inside a parenthesis in $\phi(T)$. Suppose $\lvert r(c_x) - r(c_y) \rvert = 1$ and assume $r(c_x) = i$ and $r(c_y) = i + 1$. If we let $b$ be the box right above $c_y$ in the $i$-th row, the label of $b$ must be equal to $M$ by the balancedness at $b$. This is impossible because the labelling is column-strict.

To show that $\phi$ is a bijection, we construct the inverse map $\psi$ from a cyclically decreasing factorization to a column-strict balanced labelling. Given a cyclically decreasing factorization $w = v_1 v_2 \cdots v_q$ take any cyclically decreasing reduced decomposition of $v_i$ for each $i$ and multiply them to get a reduced decomposition of $w$, e.g. $w =  (s_a s_b s_c) (s_d) (\id) (s_e s_f) \cdots$. By Theorem~\ref{thm:canonical_labelling}, this reduced decomposition corresponds to a unique injective labelling of $D(w)$. Now change the labels in the injective labelling so that the labels correponding to simple reflections in the $k$-th parenthesis will have the same label $k$, for example if $w =  (s_a s_b s_c) (s_d) (\id) (s_e s_f) \cdots$ then change the labels $\{1, 2, 3\}$ to $\{1, 1, 1\}$, $\{4\}$ to $\{2\}$, $\{5, 6\}$ to $\{4\}$ and so on. The resulting labelling is defined to be the image of the given cyclically decreasing factorization under $\psi$. It is easy to see that this labelling is also balanced so we need to show that this is labelling is column-strict and that the map is well-defined, because cyclically decreasing decomposition of an affine permutation is not unique.

Given any labelling $M$, suppose we are at the point at which we have removed all the boxes with labels greater than $M$ during the above procedure, and suppose that there are two boxes $c_x, c_y$ of the same label $M$ in the same column $j$, where $c_x$ is below $c_y$. These two boxes must be removed before we remove any other boxes with labels less than $M$, so to make $c_y$ a border cell, every boxes between $c_x$ and $c_y$ (including $c_x$) should be removed before $c_y$ gets removed. This implies that every box between $c_x$ and $c_y$ has label $M$. Let $i$ be the row index of $c_x$. Then the box $(i-1, j)$ should also have the label $M$ and it gets removed after the box $c_x$ is removed. This implies that the index $i - 1$ preceded $i$ inside a parenthesis in the original reduced decomposition, which contradict the fact that each parenthesis came from a cyclically decreasing decomposition. Thus the image of $\psi$ is column-strict.

Finally, we show that the map $\psi$ is well-defined. One easy fact from affine symmetric group theory is that any two cyclically decreasing decomposition of a given affine permutation can be obtained from each other via applying commuting relations only. Thus it is enough to show that the column-strict labellings coming from two reduced decompositions $(\cdots)\cdots(\cdots s_i s_j \cdots) \cdots (\cdots)$ and $(\cdots)\cdots(\cdots s_j s_i \cdots) \cdots (\cdots)$ coincides if $\lvert i - j \rvert \geq 2$. This is straightforward because the operation of switching the pairs of rows $(i + rk, i + 1 + rk)$, $k \in \ZZ$ is disjoint from the operation of switching the pairs of rows $(j + rk, j + 1 + rk)$, $k \in \ZZ$.

From Theorem~\ref{thm:canonical_labelling} and from the construction of $\phi$ and $\psi$, one can easily see that $\phi$ and $\psi$ are inverses of each other. This gives the desired bijection.
\end{proof}

\bigskip

\subsection{Encoding and Decoding of Reduced Decompositions}~\label{sec:encoding_decoding}

In this section we present a direct combinatorial formula for decoding reduced words from injective balanced labellings of affine permutation diagrams. Again, the theorem in \cite{Fomin1997} extends to the affine case naturally.

\begin{definition}
Let $T$ be an injective balanced labelling of $D(w)$, where $w \in \ASigma_n$ has length $\ell$. For each $k = 1, 2, \ldots, \ell$, let $\alpha_k$ be the box in $[D(w)]$ labelled by $k$, and let
\begin{eqnarray*}
&I(k) &:=\text{ the row index of }\alpha_k,\\
&R^{+}(k) &:=\text{ the number of entries }k'>k\text{ in the same row of }\alpha_k,\\
&U^{+}(k) &:=\text{ the number of entries }k'>k\text{ above }\alpha_k\text{ in the same column.}
\end{eqnarray*}
\end{definition}

\begin{theorem}\label{thm:encoding_decoding}
Let $T$ be an injective balanced labelling of $D(w)$, where $w \in \ASigma_n$ has length $\ell$, and let $a = a_1a_2\cdots a_\ell$ be the reduced word of $w$ whose canonical labelling is $T$. Then, for each $k = 1, 2, \ldots, \ell$,
$$a_k = I(k) + R^{+}(k) - U^{+}(k).$$
\end{theorem}
\begin{proof}
Our claim is that
$$I(k) = a_k + U^+(k) - R^+(k).$$
We will show that this formula is valid for all $k$ by induction on $\ell$. The formula is obvious if $\ell = 0$ or $1$.

Let $\hat{w} = w s_{a_\ell}$ so that $\hat{w}$ has length $\ell - 1$. The above formula holds for $ \hat{a} = a_1a_2 \cdots a_{\ell-1}$, i.e.
$$\hat{I}(k) = a_k + \hat{U}^+(k) - \hat{R}^+(k) \mod n,$$
where the hatted expressions correspond to the word $\hat{a}$. We now analyze the change in the quantities on the left-hand and right-hand side of our claim.

\begin{enumerate}[(1)]
\item
If $k = \ell$, then $U^+(k) = R^+(k) = 0$ and obviously $I(\ell) = a_\ell$.
\item
If $k < \ell$ and $k$ does not occur in rows $a_\ell$ or $a_\ell + 1$ of $D(\hat(w))$, then none of the quantities change.
\item
If $k < \ell$ and $k$ occur in row $a_\ell$, then $I(k) = \hat{I}(k) + 1$ and $R^+(k) = \hat{R}^+(k)$. Note that the entry $k'$ right below $k$ in $D(\hat{w})$ is greater than $k$ by Lemma~\ref{lem:localization} and it will move up when we do the exchange $s_{a_\ell}$. Thus $U^{+}(k) = \hat{U}^+(k) + 1$, and the changes on the two sides of the equation match.
\item
If $k < \ell$ and $k$ occur in row $a_\ell + 1$, then $I(k) = \hat{I}(k) - 1$ and $R^+(k) = \hat{R}^+(k) + 1$. Note that the entry $k'$ right above $k$ in $D(\hat{w})$ is less than $k$ by Lemma~\ref{lem:localization} so it did not get counted in $\hat{U}^{+}(k)$. Thus $U^{+}(k) = \hat{U}^+(k)$, and the changes on the two sides of the equation match.

\end{enumerate}
\end{proof}

\begin{remark}~\label{rem:reverse_word}
For a reduced word $a = a_1a_2 \cdots a_\ell$ of $w$ and the corresponding canonical labelling $T_a$, let $a^{-1}$ be the reversed reduced word $a_\ell a_{\ell-1} \cdots a_1$ of $w^{-1}$. It is not hard to see that the canonical labelling $T_{a^{-1}}$ corresponding to $a^{-1}$ can be obtained by taking the reflection of $T_a$ with respect to the diagonal $y = x$ and then reversing the order of the labels by $i\mapsto \ell + 1 - i$. This implies that
$$a_k = J(k) + C^-(k) - L^-(k)$$
where
\begin{eqnarray*}
&J(k) &:=\text{ the row index of }\alpha_k,\\
&C^{-}(k) &:=\text{ the number of entries }k'<k\text{ in the same column of }\alpha_k,\\
&L^{-}(k) &:=\text{ the number of entries }k'<k\text{ to the left }\alpha_k\text{ in the same row.}
\end{eqnarray*}
With careful examination one can show that the equation $I(k) + R^{+}(k) - U^{+}(k) = J(k) + C^-(k) - L^-(k)$ is equivalent to the balanced condition.
\end{remark}

\bigskip


\section{Set-Valued Balanced Labellings}
Whereas Schubert polynomials are representatives for the cohomology of the flag variety, Grothendieck polynomials are representatives for the K-theory of the flag variety. In the same
way that Stanley symmetric functions are stable Schubert polynomials, one can define stable Grothendieck polynomials as a stable limit of Grothendieck polynomials. Furthermore, Lam \cite{Lam2006} generalized this notion to the \emph{affine stable Grothendieck polynomials} and showed that they are symmetric functions. In this section we define a notion of \emph{set-valued (s-v) balanced labellings} of an affine permutation diagram and show that affine stable Grothendieck polynomials are the generating functions of column-strict s-v balanced labellings. Note that every result in this section can be applied to the usual stable Grothendieck polynomials if we restrict ourselves to the diagram of finite permutations. This can be seen as a generalization of \emph{set-valued tableaux} of Buch \cite{Buch2002} which he defined to give a formula for stable Grothendieck polynomials indexed by $321$-avoiding permutations (in other words, skew diagrams $\lambda / \mu$ where $\lambda$ and $\mu$ are partitions.)

\subsection{Set-Valued Labellings}

Let $w$ be an affine permutation and let $D(w)$ be its diagram. A \emph{set-valued (s-v) labelling} of $D(w)$ is a map $T:D(w) \rightarrow 2^{\ZZ_{>0}}$ from the boxes of $D(w)$ to subsets of positive integers such that $T(i,j) = T(i+n, j+n)$. The \emph{length} $\lvert T \rvert$ of a labelling $T$ is the sum of the cardinalities $\sum \lvert T(b) \rvert$ over all boxes $b\in [D(w)]$ in the fundamental window.

A s-v labelling $T$ is called \emph{injective} if
$$\bigcup_{b\in [D(w)]} T(b) = \{1, 2, \ldots, \lvert T \rvert\}$$
(hence the union is necessarily a disjoint union.) $T$ is called \emph{column-strict} if for any two distinct boxes $a$ and $b$ in the same column of $D(w)$, $T(a)\cap T(b) = \emptyset$.

\begin{definition}
For a box $a\in D(w)$ let $H_a$ be the hook at $a$ as before. Let $\{b_i\}_{i\in I}$ be the boxes in the right-arm of $a$ and let $\{c_j\}_{j\in J}$ be the boxes in the bottom-arm of $a$. Let $\rmin_a := \min \{\bigcup_i T(b_i)\}$ and $\bmin_a := \min \{\bigcup_j T(c_j)\}$ where $\min \emptyset := \infty$. In each box in $H_a$, we are allowed to pick one label from the box under the following conditions:
\begin{enumerate}[(1)]
\item
in box $a$, we may pick any element in $T(a)$,
\item
in box $b_i$, we may pick $\min T(b_i)$ or any element $x\in T(b_i)$ such that $x\leq \bmin_a$,
\item
in box $c_j$, we may pick $\min T(c_j)$ or any element $y\in T(c_j)$ such that $y\leq \rmin_a$.
\end{enumerate}

An s-v hook $H_a$ is called \emph{balanced} if the hook is balanced (in the sense of Definition~\ref{def:bal_hook}) for every choice of a label in each box under the above conditions.
\end{definition}

\begin{definition}
Let $w = [w_1, w_2, w_3]$ be a permutation in $\Sigma_3$. A s-v labelling $T$ of $D(w)$ is called \emph{balanced} if every hook in $D(w)$ is balanced.
\end{definition}

\begin{definition}[S-V Balanced Labellings]~\label{def:sv_balanced}
Let $w$ be any affine permutation and let $T$ be a s-v labelling of $D(w)$. $T$ is called \emph{balanced} if the $3\times 3$ subdiagram $D_{i,j,k}$ determined by rows $i,j,k$ and columns $w(i), w(j), w(k)$ is balanced for every $i<j<k$. (cf. Lemma~\ref{lem:localization}.)
\end{definition}

Note that when $w$ is a $321$-avoiding finite permutation, Definition~\ref{def:sv_balanced} is equivalent to the \emph{set-valued tableaux} of Buch \cite{Buch2002}.

\begin{lemma}~\label{lem:sv_global}
If $T$ is a s-v balanced labelling, then every hook of $T$ is balanced.
\end{lemma}
\begin{proof}
The first half of the proof of Lemma~\ref{lem:localization} will work here if one replaces single-valued labels with set-valued labels.
\end{proof}

\begin{figure}
\centering
\includegraphics{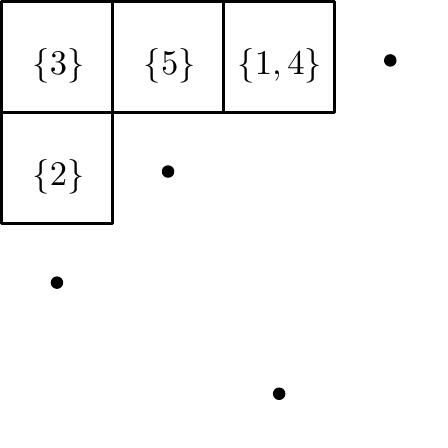}
\caption{non-balanced s-v diagram}
\label{fig:sv_diag}
\end{figure}

\begin{remark}
If every label set consists of a single element, then Definition~\ref{def:sv_balanced} is equivalent to the original definition of (single-valued) balanced labellings by Lemma~\ref{lem:localization}. One may wonder why we must take this \emph{local} definition of checking all the $3\times 3$ subdiagrams rather than simply requiring that every hook in the diagram is balanced \emph{globally} as we did for single-valued diagrams. Lemma~\ref{lem:sv_global} shows that that the global definition is weaker than the local definition in the set-valued case and, in fact, it is strictly weaker. Figure~\ref{fig:sv_diag} is an example of a diagram in which every hook is balanced globally but it is not balanced in our definition if we take the subdiagram determined by rows $1,3,4$. We will show in the following sections that this local definition is the ``right'' definition for s-v balanced labellings.
\end{remark}

\subsection{NilHecke Words and Canonical S-V Labellings}

Whereas a balanced labelling is an encoding of a reduced word of an affine permutations, a s-v balanced labelling is an encoding of a \emph{nilHecke word}. Let us recall the definition of the affine nilHecke algebra. An affine nilHecke algebra $\AUU_n$ is generated over $\ZZ$ by the generators $u_0, u_1, \ldots, u_{n-1}$ and relations
\begin{eqnarray*}
&u_i^2=u_i &\text{for all } i \\
&u_i u_{i+1} u_i = u_{i+1} u_i u_{i+1} &\text{for all } i \\
&u_i u_j = u_j u_i &\text{for }|i-j|\geq 2
\end{eqnarray*}
where indices are taken modulo $n$. A sequence of indices $a_1, a_2, \ldots a_k \in [0, n-1]$ is called a \emph{nilHecke word} and it defines an element $u_{a_1} u_{a_2} \cdots u_{a_k}$ in $\AUU_n$. $\AUU_n$ is a free $\ZZ$-module with basis $\{u_w \mid w\in \ASigma_n\}$ where $u_w = u_{i_1} u_{i_2} \cdots u_{i_\ell}$ for any reduced word $(i_1, i_2, \ldots, i_\ell)$ of $w$. The multiplication under this basis is given by
$$u_i u_w = \begin{cases}
  u_{s_i w}& \text{if $i$ is not a descent of $w$,}\\
  u_{w}& \text{if $i$ is a descent of $w$.}
\end{cases}$$
Note that for any nilHecke word $a_1,a_2, \ldots, a_k$ in $\AUU_n$, there is a unique affine permutation $w\in \ASigma_n$ such that $u_w = u_{a_1} u_{a_2} \cdots u_{a_k}$. In this case we denote $S(a_1, a_2, \ldots, a_k) = w$.

\begin{definition}[Canonical s-v labelling] Let $w\in \ASigma_ n$ be an affine permutation and let $a = (a_1, a_2, \ldots, a_k)$ be a nilHecke word in $\AUU_n$ such that $S(a) = w$. Let $w' = S(a')$ where $a' = (a_1, a_2, \ldots, a_{k-1})$. Define a s-v injective labelling $T_a:D(w)\rightarrow 2^{\{1,\cdots,k\}}$ recursively as follows.
\begin{enumerate}[(1)]
\item If $a_k$ is a decent of $w'$, then $D(w) = D(w')$. Add a label $k$ to the sets $T_{a'}(a_k + rn, w'(a_k + 1) + rn))$, $r\in \ZZ$.
\item If $a_k$ is not a decent of $w'$, then $D(w)$ is obtained from $D(w')$ by switching the pairs of rows $(a_k + rn, a_k + 1 + rn)$, $r\in \ZZ$ and adding a cell $\cell{(a_k, w(a_k+1))}$. Label the newly appeared boxes $(a_k + rn, w(a_k + 1) + rn)$, $r\in \ZZ$, by a single element set $\{k\}$.
\end{enumerate}
We call $T_a$ the \emph{canonical s-v labelling} of $a$.
\end{definition}

The following results are set-valued generalizations of Proposition~\ref{prop:canonical_balanced}, Lemma~\ref{lem:max_label}, and Theorem~\ref{thm:removing_border_balanced}.

\begin{proposition}~\label{prop:sv_canonical_balanced}
Let $w\in \ASigma_n$. A canonical labelling of a nilHecke word $a$ in $\AUU_n$ with $S(a) = w$ is a s-v injective balanced labelling of $D(w)$.
\end{proposition}
\begin{proof}
We show that for any triple $i < j < k$ the intersection $D_{ijk}$ of the canonical labelling of $a$ with the rows $i, j, k$ and the columns $w(i), w(j), w(k)$ is balanced. Let $p < q < r$ be the rearrangement of $w(i), w(j), w(k)$.

If $w(j) < w(k) < w(i)$ or $w(k) < w(i) < w(j)$ so that there are two boxes in $D_{ijk}$, then the same arguments we used in the proof of Proposition~\ref{prop:canonical_balanced} will work. If $w(k) < w(j) < w(i)$ so that there are three boxes in $D_{ijk}$ in ``$\Gamma$''-shape, then $w = [\ldots, r, \ldots, q, \ldots, p, \ldots]$ in this case. If $p$ passes through $q$ before $r$ passes through $q$, then $r$ should pass through $p$ before it passes through $q$. This implies that every label in the box $(j,p)$ which is less than the minimal label of $(i,q)$ is less than any label of $(i,p)$. Also, every label of $(i,q)$ is larger than any label of $(i,p)$. Hence, $D_{ijk}$ is balanced. A similar argument will work for the case where $r$ passes through $q$ before $p$ passes through $q$.
\end{proof}

\begin{lemma}\label{lem:sv_max_label}
Let $T$ be a s-v column-strict balanced labelling of $D(w)$ with largest label $M$, then every row containing an $M$ must contain an $M$ in a border cell. In particular, if $i$ is the index of such row, then $i$ must be a descent of $w$. Futhermore, if a border cell containing $M$ contains two or more labels, then it must be the only cell in row $i$ which contains an $M$.
\end{lemma}
\begin{proof}
Suppose that the row $i$ contains a label $M$. First we show that $i$ is a descent of $w$. If $i$ is not a descent, i.e. $w(i) < w(i+1)$, then let $(i,j)$ be the rightmost box in row $i$ whose label set contains $M$. By the balancedness of the subdiagram $D_{i, i+1, w^{-1}(j)}$, labels of the box $(i+1, j)$ must be greater than $M$, which is a contradiction. Therefore, $i$ is a decent.

Let $w(i+1) = j$, i.e. $(i,j)$ is a border cell. We must show that $M\in T(i,j)$. If every label of $T(i,j)$ is less than $M$, then a label $M$ cannot occur in the right-arm of $(i,j)$ by the balancedness. Let $(i, k)$, $k < j$, be the rightmost occurrence of $M$ in the $i$-th row. Then the subdiagram $D_{i,i+1, w^{-1}(k)}$ is not balanced.

For the last sentence of the lemma, let $(i,j)$ be a border cell such that $M\in T(i,j)$ and $\lvert T(i,j) \rvert \geq 2$. One can follow the argument in the previous paragraphs to show that there cannot be an occurence of $M$ to the right of $(i,j)$ and to the left of $(i,j)$ in row $i$.
\end{proof}

\begin{definition}
Given a s-v column-strict balanced labelling $T$ with largest label $M$, a border cell containing $M$ is called a \emph{type-I maximal cell} if it has a single label $M$, and \emph{type-II maximal cell} if it contains more than one labels.
\end{definition}

\begin{theorem}\label{thm:sv_removing_border_balanced}
Let $T$ be a s-v column-strict labelling of $D(w)$, and let $\alpha$ be a border cell containing the largest label $M$ in $T$. Let $T\setminus \alpha$ be the s-v labelling we obtain from $T$ as follows: If $\alpha$ is a type-II maximal cell, then simply delete the label $M$ from the label set of $\alpha$. If $\alpha$ is a type-I maximal cell, then delete all the boxes of $\alpha$ and switch pairs of rows $(i + rn, i + 1 + rn)$ for all $r\in \ZZ$ from $T$. 
Then $T$ is balanced if and only if $T \setminus \alpha$ is balanced.
\end{theorem}
\begin{proof}
This is a routine verification following the arguments we used for the proof of Theorem~\ref{thm:removing_border_balanced}. Definition~\ref{def:sv_balanced} replaces Lemma~\ref{lem:localization} in the set-valued case. Note that removing the largest label $M$ from a type-II maximal cell does not affect the balancedness of the diagram.
\end{proof}

Now we present the main theorem of this section.

\begin{theorem}~\label{thm:sv_canonical_labelling}
Let $w\in \ASigma_ n$ be an affine permutation. The map $a \mapsto T_a$ is a bijection from the set of all nilHecke words $a$ in $\AUU_n$ with $S(a) = w$ to the set of all s-v injective balanced labellings of $D(w)$.
\end{theorem}
\begin{proof}
This is a direct consequence of Lemma~\ref{lem:removing_border_cell}, Proposition~\ref{prop:sv_canonical_balanced}, Lemma~\ref{lem:sv_max_label}, and Theorem~\ref{thm:sv_removing_border_balanced}.
\end{proof}

As in the case of single-valued labellings, we have a direct formula for decoding nilHecke words from s-v injective balanced labellings. The following theorem is a set-valued generalization of Theorem~\ref{thm:encoding_decoding}

\begin{theorem}~\label{thm:sv_encoding_decoding}
Let $T$ be a s-v injective balanced labelling of $D(w)$ with $\lvert T \rvert = k$, $w\in \ASigma_n$. For each $t = 1, 2, \ldots, k$, let $\alpha_t$ be the box in $[D(w)]$ labelled by $t$ and define $I(t)$, $R^{+}(t)$, and $U^+(t)$ as follows.
\begin{eqnarray*}
I(t) &:= &\text{the row index of }\alpha_t.\\
R^{+}(t) &:= &\text{the number of boxes in the same row of }\alpha_t,\\
&&\text{whose minimal label is greater than } t.\\
U^{+}(t) &:= &\text{the number of boxes above }\alpha_t\text{ in the same column,}\\
&&\text{whose minimal label is greater than } t.
\end{eqnarray*}
Let $a = (a_1, a_2, \ldots, a_k)$ be the nilHecke word whose canonical labelling is $T$. Then, for each $t = 1, 2, \ldots, k$,
$$a_t = I(t) + R^{+}(t) - U^{+}(t)\mod n.$$
\end{theorem}

\begin{proof}
Our claim is that $I(t) = a_t + U^+(t) - R^+(t)\mod n$.
We will show that this formula is valid for all $t$ by induction on $k$. The formula is obvious if $k = 0$ or $1$.

Let $\hat{a} = (a_1,a_2, \ldots, a_{k-1})$. If $a_k$ is a descent of $S(\hat{a})$, then $S(\hat{a}) = w = S(a)$ and $T_a$ is obtained from $T_{\hat{a}}$ by simply adding the largest label $k$ to the (already existing) border cell in the $a_k$-th row. In this case, it is clear that $I(t)$, $R^{+}(t)$, and $U^{+}(t)$ stays the same for $t = 1, 2, \ldots, k-1$ and that $a_k = I(k)$, so the formula holds by induction.

Now suppose $a_k$ is not a descent of $S(\hat{a})$ so $S(\hat{a}) = w s_{a_k} =: \hat{w}$. Again by induction, the above formula holds for $ \hat{a}$ so
$$\hat{I}(t) = a_t + \hat{U}^+(t) - \hat{R}^+(t) \mod n,$$
where the hatted expressions correspond to the labelling $T_{\hat{a}}$. We now analyze the change in the quantities on the left-hand side and the right-hand side of our claim.

\begin{enumerate}[(1)]
\item
If $t = k$, then $U^+(k) = R^+(k) = 0$ and obviously $I(k) = a_k$.
\item
If $t < k$ and $t$ does not occur in rows $a_k$ or $a_k + 1$ of $D(\hat{w})$, then none of the quantities change.
\item
If $t < k$ and $t$ occurs in row $a_k$, then $I(t) = \hat{I}(t) + 1$ and $R^+(t) = \hat{R}^+(t)$. Note that the minimal entry $t'$ of the box right below $t$ in $D(\hat{w})$ is greater than $t$ and it will move up when we do the exchange $s_{a_k}$. Thus $U^{+}(t) = \hat{U}^+(t) + 1$, and the changes on the two sides of the equation match.
\item
If $t < k$ and $t$ occurs in row $a_k + 1$, then $I(t) = \hat{I}(t) - 1$ and $R^+(t) = \hat{R}^+(t) + 1$. Note that the minimal entry $t'$ of the box right above $t$ in $D(\hat{w})$ is less than $t$ so it did not get counted in $\hat{U}^{+}(t)$. Thus $U^{+}(t) = \hat{U}^+(t)$, and the changes on the two sides of the equation match.
\end{enumerate}
\end{proof}

\subsection{Affine Stable Grothendieck Polynomials}
An affine stable Grothendieck polynomial of Lam \cite{Lam2006} can be defined in terms of words in \emph{affine nilHecke algebra} (see also \cite{Lam2009} and \cite{Morse2012}).

Let $w$ be an affine permutation in $\ASigma_n$. A \emph{cyclically decreasing nilHecke factorization} $\alpha$ of $w$ is a factorization $u_w = u_{v_1} u_{v_2} \cdots u_{v_k}$ where each $v_i$ is a cyclically decreasing affine permutation in $\ASigma_n$. The sequence $(\ell(v_1), \ell(v_2), \ldots, \ell(v_k))$ is called the \emph{type} of $\alpha$. Let $\lvert \alpha \rvert := \ell(v_1)+\ell(v_2) + \cdots +  \ell(v_k)$. The \emph{affine stable Grothendieck polynomial} $G_w$ is defined by
$$\widetilde{G}_w (x) = \sum_{\alpha} (-1)^{\lvert \alpha \rvert - \ell(w)} x_1^{\ell(v_1)} x_2^{\ell(v_2)} \cdots x_k^{\ell(v_k)},$$
where the sum is over all cyclically decreasing \emph{nilHecke} factorization $\alpha: u_w = u_{v_1}u_{v_2} \cdots u_{v_k}$ of $w$. Note that this function is a generalization of the usual stable Grothendieck polynomial and that its minimal degree terms ($\lvert \alpha \rvert = \ell(w)$) form the affine Stanley symmetric function. Lam \cite{Lam2006} showed that this function is a symmetric function.

In this section, we show that affine stable Grothendieck polynomials are the generating functions of the column-strict s-v balanced labellings.

\begin{theorem}~\label{thm:affine_stable_grothendieck}
Let $w\in \ASigma_n$ be an affine permutation. Then
$$\widetilde{G}_w(x) = \sum_{T} (-1)^{\lvert T \rvert - \ell(w)} x^T,$$
where the sum is over all \emph{column-strict} s-v balanced labellings $T$ of $D(w)$, and $x^T$ is the monomial $\prod_{b\in [D(w)]} \prod_{k\in T(b)} x_k$.
\end{theorem}

Before we give a proof of the theorem, we state a general fact about column-strict s-v balanced labellings.

\begin{lemma}~\label{lem:sv_no_M_row}
Let $T$ be a column-strict s-v balanced labelling of $D(w)$ where $w\in \ASigma_n$. Let $M$ be the largest label of $T$. Then, there exists $p\in \{1, 2, \ldots, n\}$ such that there is no label $M$ in the $p$-th row of $T$.
\end{lemma}

\begin{proof}
By the column-strictness and the periodicity of the diagram, there can be at most $n$ $M$'s in the fundamental window $\{1, 2, \ldots, n\} \times \ZZ$. If the number of $M$'s in the fundamental window is less than $n$, then the lemma is true.

Suppose the number of $M$'s in the fundamental windows is exactly $n$. If there is a row containing two or more $M$'s, then again the proof follows. If each row $p\in \{1, 2, \ldots, n\}$ contains exactly one $M$, then by Lemma~\ref{lem:sv_max_label} every $p$ is a descent, which is impossible.
\end{proof}

\begin{proof}[Proof of Theorem~\ref{thm:affine_stable_grothendieck}]
Given a column-strict s-v balanced labelling $T$, we call the sequence $([$the number of $1$'s in $T]$, $[$the number of $2$'s in $T]$, $\ldots )$ the \emph{type} of the labelling. It is enough to show that there is a type-preserving bijection $\phi$ from a column-strict s-v labelling of $D(w)$ to a cyclically decreasing nilHecke factorization of $w$.

Let us construct $\phi$ as follows. Given a column-strict s-v labelling $T$ with $t = \lvert T \rvert$, let $M$ be its largest label. If $T$ has a \emph{type-I} maximal cell, then let $c_1$ to be any of those type-I maximal cells. If all the border cells with label $M$ of $T$ is type-II, then let $c_1$ to be a maximal cell in some row $i$ such that there is no $M$ in the $i-1$-st row (by Lemma~\ref{lem:sv_no_M_row}). Let $r(c_1)$ be the row index of $c_1$ in the fundamental window. By Theorem~\ref{thm:sv_removing_border_balanced}, we obtain a column-strict s-v balanced labelling $T\setminus c_1$ by (1) removing the cell $c_1$ and switching all pairs of rows $(r(c_1) + kn, r(c_1) + kn + 1)$ for all $k \in \ZZ$ if $c_1$ is type-I, or (2) simply removing $M$ from the label set of $c_1$ if $c_1$ is type-II. The resulting labelling $T\setminus c_1$ is a labelling of length $t - 1$ of the diagram of the affine permutation $w s_{r(c_1)}$ in case (1), or of $w$ in case (2). In $T\setminus c_1$, we again pick a maximal cell $c_2$ by the same procedure (by Theorem~\ref{thm:sv_removing_border_balanced}) and obtain the labelling $T\setminus c_1 \setminus c_2$ of length $t-2$. We continue this process removing labels in cells $c_1, c_2, \ldots, c_t$ until we get the empty digram which corresponds to the identity permutation. Then, $r(c_t), r(c_{t-1}), \ldots, r(c_1)$ is a nilHecke word such that $w = S(r(c_t), r(c_{t-1}), \ldots, r(c_1))$. Now in this nilHecke word, group the terms together in the parenthesis if they correspond to removing the same largest label of the digram in the process and this will give you a factorization of $u_w$. With careful examination, one can see that words in the same parenthesis is cyclically decreasing so this gives a cyclically decreasing nilHecke factorization of $w$ corresponding to $T$ under $\phi$.

Now we show that $\phi$ is well-defined regardless of the choice of $c_i$'s in the process. It is enough to show that if we had a choice of taking one of the two border cells $c_x$ and $c_y$ with the same largest labelling at some point, then $\lvert r(c_x) - r(c_y) \rvert \geq 2$ so the corresponding simple reflections commute inside a parenthesis in $\phi(T)$. Suppose $\lvert r(c_x) - r(c_y) \rvert = 1$ and assume $r(c_x) = i$ and $r(c_y) = i + 1$. By construction, this can only happen when both $c_x$ and $c_y$ are type-I maximal cells. If we let $b$ be the box right above $c_y$ in the $i$-th row, the label of $b$ must be equal to $M$ by the balancedness at $b$. This is impossible because the labelling is column-strict.

To show that $\phi$ is a bijection, we construct the inverse map $\psi$ from a cyclically decreasing nilHecke factorization to a column-strict s-v balanced labelling. Given a cyclically decreasing nilHecke factorization $u_w = u_{v_1} u_{v_2} \cdots u_{v_q}$, take any cyclically decreasing reduced decomposition of $v_i$ for each $i$ inside a parenthesis, and then their concatenation is a nilHecke word which multiplies to $u_w$. By Theorem~\ref{thm:sv_canonical_labelling}, this nilHecke word corresponds to a unique injective s-v labelling of $D(w)$. Now change the labels in the injective s-v labelling so that the labels corresponding to $u_i$'s in the $k$-th parenthesis will have the same label $k$. The resulting s-v labelling is defined to be the image of the given cyclically decreasing nilHecke factorization under $\psi$. It is easy to see that this s-v labelling is also balanced so it remains to show that this s-v labelling is column-strict and that the map is well-defined.

Given any label $M$, suppose we are at the point at which we have removed all the labels greater than $M$ during the above procedure, and suppose that there are two boxes $c_x, c_y$ which contains the same label $M$ in the same column $j$, where $c_x$ is below $c_y$. These two boxes must be removed before we remove any other boxes with labels less than $M$, so to make $c_y$ a border cell, every boxes between $c_x$ and $c_y$ (including $c_x$) should be removed before $c_y$ gets removed. This implies that every box between $c_x$ and $c_y$ has a single label $\{M\}$. Let $c_x = (i,j)$. Then the box $(i-1, j)$ should also have a label $M$ and it gets removed after the box $c_x$ is removed. This implies that the index $i - 1$ preceded $i$ inside a parenthesis in the original nilHecke word, which contradict the fact that each parenthesis came from a cyclically decreasing decomposition. Thus the image of $\psi$ is column-strict.

Finally, we show that the map $\psi$ is well-defined. One easy fact from affine symmetric group theory is that any two cyclically decreasing decomposition of a given affine permutation can be obtained from each other via applying commuting relations only. Thus it is enough to show that the column-strict labellings coming from two reduced decompositions $(\cdots)\cdots(\cdots u_i u_j \cdots) \cdots (\cdots)$ and $(\cdots)\cdots(\cdots u_j u_i \cdots) \cdots (\cdots)$ coincides if $\lvert i - j \rvert \geq 2$ modulo $n$. This is straightforward because the operation of switching the pairs of rows $(i + rk, i + 1 + rk)$, $k \in \ZZ$ is disjoint from the operation of switching the pairs of rows $(j + rk, j + 1 + rk)$, $k \in \ZZ$.

From Theorem~\ref{thm:sv_canonical_labelling} and from the construction of $\phi$ and $\psi$, one can easily see that $\phi$ and $\psi$ are inverses of each other. This gives the desired bijection.
\end{proof}
\bigskip

\subsection{Grothendieck Polynomials}

Let us restrict our attention to finite permutations $w\in \Sigma_n$ for this section. In this case, there is a type-preserving bijection from column-strict labellings to \emph{decreasing} nilHecke factorizations of $w$, i.e., $u_w = u_{v_1} u_{v_2} \cdots u_{v_k}$ in the nilHecke algebra $\UU_n = \langle u_1, u_2, \ldots, u_{n-1} \rangle$, where each $v_i$ are permutations having \emph{decreasing} reduced word. Theorem~\ref{thm:affine_stable_grothendieck} reduces to a monomial expansion of the stable Grothendieck polynomial $G_w(x)$ in terms of column-strict s-v labellings of (finite) Rothe diagram of $w$.

Let $\mathfrak{G}_w(x)$ be the Grothendieck polynomial of Lascoux-Sch\"{u}tzenberger \cite{Lascoux1983}. Fomin-Kirillov \cite{Fomin1993} showed that 
\begin{equation}~\label{eqn:grothendick}
\mathfrak{G}_w (x) = \sum_{\beta} (-1)^{\lvert \alpha \rvert - \ell(w)} x_1^{\ell(v_1)} x_2^{\ell(v_2)} \cdots x_k^{\ell(v_k)},
\end{equation}
where the sum is over all \emph{flagged} decreasing nilHecke factorization $\beta: u_w = u_{v_1}u_{v_2} \cdots u_{v_k}$ of $w$, i.e., each $v_i$ has a decreasing reduced word $a_1 a_2 \cdots a_{\ell(v_i)}$ such that $a_j \geq i$ for all $j$.

We show in this section that this formula leads to another combinatorial expression for $\mathfrak{G}_w$ involving just a single sum over column-strict s-v balanced labellings with \emph{flag conditions}.

\begin{theorem}~\label{thm:grothendieck}
Let $w\in \Sigma_n$ be a finite permutation. Then
$$\mathfrak{G}_w(x) = \sum_{T} (-1)^{\lvert T \rvert - \ell(w)} x^T,$$
where the sum is over all \emph{column-strict} s-v balanced labellings $T$ of $D(w)$ such that for every label $t\in T(i,j)$, $t\leq i$.
\end{theorem}

The content of Theorem~\ref{thm:grothendieck} is that the flag condition in (\ref{eqn:grothendick}) translates to the flag condition $t \leq i$, $\forall t\in T(i,j)$. To be precise, the following lemma implies Theorem~\ref{thm:grothendieck}. (Note that the sequence $i_1 \leq i_2 \leq \cdots \leq i_k$ in the lemma corresponds to the column-strict labels we construct in the proof of Theorem~\ref{thm:affine_stable_grothendieck}.)

\begin{lemma}
Suppose that $a = (a_1, a_2, \ldots, a_k)$ is a nilHecke word in $\UU_n$ and let $T_a$ be a s-v balanced labelling corresponding to $a$. Let $i_1. i_2. \ldots, i_k$ be a sequence of positive integers satisfying $i_1 \leq i_2 \leq \cdots \leq i_k$. Then,
\begin{equation}~\label{eqn:flag1}
i_t \leq a_t
\end{equation}
holds for all $t = 1, 2, \ldots, k$ if and only if
\begin{equation}~\label{eqn:flag2}
i_t \leq I(t)
\end{equation}
holds for all $t = 1, 2, \ldots, k$. As before, $I(t)$ denotes the row index of the box containing the label $t$ in $T_a$.
\end{lemma}
\begin{proof}
We have $a_t = I(t) + R^+(t) - U^+(t)$ for all $t$ by Theorem~\ref{thm:sv_encoding_decoding}. Suppose (\ref{eqn:flag1}) holds. We want to show $i_t \leq I(t)$.

If $R^+(t) = 0$, then $i_t\leq a_t = I(t) - U^+(t) \leq I(t)$. If $R^+(t) >0$, then let $t' > t$ be the largest label in row $I(k)$. Clearly $R^+(t')=0$, so $i_{t'} \leq I(t')$. Thus
$$i_t \leq i_{t'} \leq I(t') = I(t).$$
This completes one direction of the lemma.

Next, suppose (\ref{eqn:flag2}) holds. We have $i_t \leq I(t) = a_t - R^+(t) + U^+(t)$ and we want to show $i_t \leq a_t$. If $U^+(t) = 0$, then the proof follows immediately. Suppose $U^+(t) = d > 0$. Then there are $d$ boxes above $t$ in the same column, whose minimal label is larger than $t$. If $t'$ be the one in the highest row, then $I(t') \leq I(t) - d$. Therefore,
$$i_t \leq i_{t'} \leq I(t') \leq I(t) - d = a_t - R^+(t) \leq a_t.$$
\end{proof}


\section{Characterization of Diagrams via Content}

One unexpected application of balanced labellings is a nice characterization of affine permutation diagrams. We will introduce the notion of the \newword{content map} of an affine diagram, which generalizes the classical notion of content of a Young diagram. We will conclude that the existence of such map, along with the \newword{North-West property}, completely characterizes the affine permutation diagrams.

\bigskip

\subsection{Content Map}

Given an affine diagram $D$ of size $n$, the \newword{oriental labelling} of $D$ will denote the injective labelling of the diagram with numbers from $1$ to $n$ such that the numbers increases as we read the boxes in $[D]$ from top to bottom, and from right to left. See Figure \ref{fig:oriental}. (This reading order reminds us the traditional way to write and read a book in some East Asian countries such as Korea, China, or Japan, and hence the term ``oriental''.)

\begin{lemma}
The oriental labelling of an affine (or finite) diagram is a balanced labelling.
\end{lemma}

\begin{proof}
It is clear that every hook in the oriental labelling will stay the same after rearrangement.
\end{proof}

Now, suppose we start from an affine permutations and we construct the oriental labelling of the diagram of the permutation. For example, let $w = [2,6,1,4,3,7,8,5] \in \Sigma_8 \subset \widetilde{\Sigma}_8$. Figure \ref{fig:oriental} shows the oriental labelling of the diagram of $w$, where the box labelled by $7$ is at the (1,1)-coordinate.

\begin{figure}[b]
\begin{minipage}[b]{0.45\linewidth}
\centering
\includegraphics{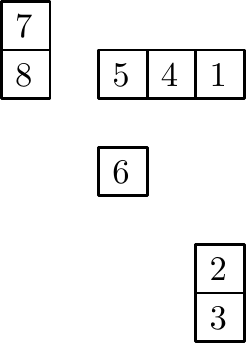}
\caption{oriental labelling of a finite diagram}
\label{fig:oriental}
\end{minipage}
\hspace{0.5cm}
\begin{minipage}[b]{0.45\linewidth}
\centering
\includegraphics{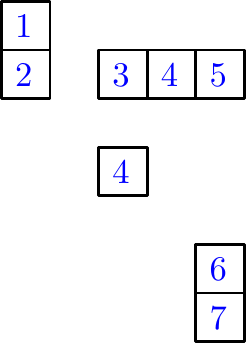}
\caption{$a_k$'s of the oriental labelling}
\label{fig:content}
\end{minipage}
\end{figure}

Following the spirit of Theorem \ref{thm:encoding_decoding}, for each box with label $k$ in the diagram, let us write down the integer $a_k$ where $a_k = I(k) + R^{+}(k) - U^{+}(k)$. Recall that $I(k)$ is the row index, $R^{+}(k)$ the number of entries greater than $k$ in the same row, $U^{+}(k)$ the number of entries greater than $k$ and located above $k$ in the same column. The formula is actually much simpler in the case of the oriental labelling, since $U^{+}(k)$ vanishes and $R^{+}(k)$ is simply the number of boxes to the left of the box labelled by $k$. Figure \ref{fig:content} illustrates the diagram filled with $a_k$ instead of $k$. From Theorem \ref{thm:encoding_decoding}, we already know that we can recover the affine permutation we started with by $a_k$'s. For example, $w = [2,6,1,4,3,7,8,5] = s_5 s_6 s_7 s_4 s_3 s_4 s_1 s_2$, where the right hand side comes from reading the Figure \ref{fig:content} ``orientally'' modulo $8$.

Motivated by this example, we define a special way of assigning integers to each box of a diagram, which will take a crucial role in the rest of this section.

\begin{definition}~\label{def:content}
Let $D$ be an affine diagram with period $n$. A map
$\mathcal{C}: D \rightarrow \ZZ$ is called a \emph{content} if it satisfies the following four conditions.
\begin{enumerate}
\item[(C1)] If boxes $b_1$ and $b_2$ are in the same row (respectively, column), $b_2$ being to the east (resp., south) to $b_1$, and there are no boxes between $b_1$ and $b_2$, then $\mathcal{C}(b_2)-\mathcal{C}(b_1) = 1$.

\item[(C2)] If $b_2$ is strictly to the southeast of $b_1$, then $\mathcal{C}(b_2) - \mathcal{C}(b_1) \geq 2$.

\item[(C3)] If $b_1 = (i,j)$ and $b_2 = (i+n, j+n)$ coordinate-wise, then $\mathcal{C}(b_2) - \mathcal{C}(b_1) = n$.

\item[(C4)] For each row (resp., column), the content of the leftmost (resp., topmost) box is equal to the row (resp., column) index.
\end{enumerate}
\end{definition}

\medskip


\begin{proposition}
Let $D$ be the diagram of an affine permutation $w \in \widetilde{\Sigma}_n$. Then, $D$ has a unique content map.
\end{proposition}

\begin{proof}
By the conditions (C1) and (C4), a content map is unique when it exists. As we have seen in Figure~\ref{fig:oriental} and Figure~\ref{fig:content}, give the oriental labelling to $D(w)$ and define $\mathcal{C}$ by $\mathcal{C}(b) := I(b) + R^{+}(b) - U^{+}(b)$ as before. In the case of the oriental labelling $R^{+}(b)$ is just the number of boxes to the left of $b$ and $U^{+}(b) = 0$. Thus
\begin{equation}\label{eqn:oriental}
\begin{split}
\mathcal{C}(b) &= (\text{row index of } b) + (\text{number of boxes to the left of } b)\\
&= (\text{column index of } b) + (\text{number of boxes above } b)
\end{split}
\end{equation}
where the second equality is from Remark~\ref{rem:reverse_word}.

(C1) is immediate for two horizontally consecutive boxes. Suppose two boxes $b_1$ and $b_2$ are in the same column, $b_2$ being to the south to $b_1$, and there are no boxes between $b_1$ and $b_2$. Let $i_1$ and $i_2$ be the row indexes of $b_1$ and $b_2$, and let $j$ be their column index. Since there are no boxes between $b_1$ and $b_2$, the dots (points corresponding to $w$) in row $i_1 + 1, i_1 + 2, \ldots, i_2 - 1$ are placed all to the left of the column $j$. These dots exactly correspond to the columns $k < j$ such that $(i_1, k)$ has a box but $(i_2, k)$ is empty. This implies that $R^{+}(b_1) - R^{+}(b_2) = i_2 - i_1 - 1$. We also have $I(b_2) - I(b_1) = i_2 - i_1$. Hence, $\mathcal{C}(b_2) - \mathcal{C}(b_1) = 1$.

For (C2), let $b_1 = (i_1, j_1)$, $b_2 = (i_2, j_2)$ be two boxes with $i_1 < i_2$ and $j_1 < j_2$, and our claim is that $\mathcal{C}(b_2) - \mathcal{C}(b_1) \geq 2$. We may assume that there are no boxes inside the rectangle $(i_1, j_1)$, $(i_1, j_2)$, $(i_2, j_1)$, $(i_2, j_2)$ since it suffices to show the claim for such pairs. Since there is no box at $(i_2, j_1)$ there must be a dot at column $j_1$ somewhere between $(i_1 + 1, j_1)$ and $(i_2 - 1, j_1)$. Hence, there are at most $i_2 - i_1 - 2$ dots to the left of column $j_1$ in rows $i_1 +1, i_1 + 2, \ldots, i_2 - 1$. This implies $R^{+}(b_1) - R^{+}(b_2) \leq i_2 - i_1 - 2$ and therefore $\mathcal{C}(b_2) - \mathcal{C}(b_1) \geq 2$.

(C3) and (C4) is clear from (\ref{eqn:oriental}).
\end{proof}

\bigskip

\subsection{Wiring Diagram and Classification of Permutation Diagrams}~

We start this section by recalling a well-known property of (affine) permutation diagrams.
\begin{definition}
An affine diagram is called \newword{North-West} (or \newword{NW)} if, whenever there is a box at $(i,j)$ and at $(k,\ell)$ with the condition $i<k$ and $j>\ell$, there is a box at $(i,\ell)$.
\end{definition}

It is easy to see that every affine permutation diagram is NW. In fact, if $(i, w^{-1}(j))$ and $(k, w^{-1}(\ell))$ is an inversion and $i<k$, $j>\ell$, then $(i, w^{-1}(\ell))$ is also an inversion since $i < k < w^{-1}(\ell)$ and $w(i) > j > \ell$. The main theorem of this section is that the content map and the NW property completely characterize the affine permutation diagrams. 

\begin{theorem}~\label{thm:diagram_classification}
An affine diagram is an affine permutation diagram if and only if it is NW and admits a content map.
\label{thm:diagchar}
\end{theorem}
\medskip

In fact, given a NW affine diagram $D$ of period $n$ with a content map, we will introduce a combinatorial algorithm to recover the affine permutation $w \in \widetilde{\Sigma}_n$ corresponding to $D$. This will turn out to be a generalization of the \newword{wiring diagram} appeared in the section 19 of \cite{Postnikov2006}, which gave a bijection between Grassmannian permutations and the partitions.
\medskip

Let $D$ be a NW affine diagram of period $n$ with a content map. A northern edge of a box $b$ in $D$ will be called a \newword{N-boundary} of $D$ if
\begin{enumerate}[(1)]
\item $b$ is the northeast-most box among all the boxes with the same content and
\item there is no box above $b$ on the same column.
\end{enumerate}
Similarly, an eastern edge of a box $b$ in $D$ will be called a \newword{E-boundary} of $D$ if
\begin{enumerate}[(1)]
\item $b$ is the northeast-most box among all the boxes with the same content and
\item there is no box to the right of $b$ on the same row.
\end{enumerate}
A northern or eastern edge of a box in $D$ will be called a \newword{NE-boundary} if it is either a N-boundary or an E-boundary. We can define an \newword{S-boundary}, \newword{W-boundary}, and \newword{SW-boundary} in the same manner by replacing ``north'' by ``south'', ``east'' by ``west'', ``above'' by ``below'', ``right'' by ``left'', etc.
\medskip

Now, from the midpoint of each NE-boundary, we draw an infinite ray to NE-direction (red rays in Figure \ref{fig:boundary}) and index the ray ``$i$'' if it is a N-boundary of a box of content $i$, and ``$i+1$'' if it is an E-boundary of a box of content $i$. We call such rays \newword{NE-rays}. Similarly, a \newword{SW-ray} is an infinite ray from the midpoint of each SW-boundary to SW-direction (blue rays in Figure \ref{fig:boundary}), indexed ``$w_i$'' if it is a W-boundary of a box of content $i$, and ``$w_{i+1}$'' if it is a S-boundary of a box of content $i$.

\begin{figure}
\centering
\includegraphics{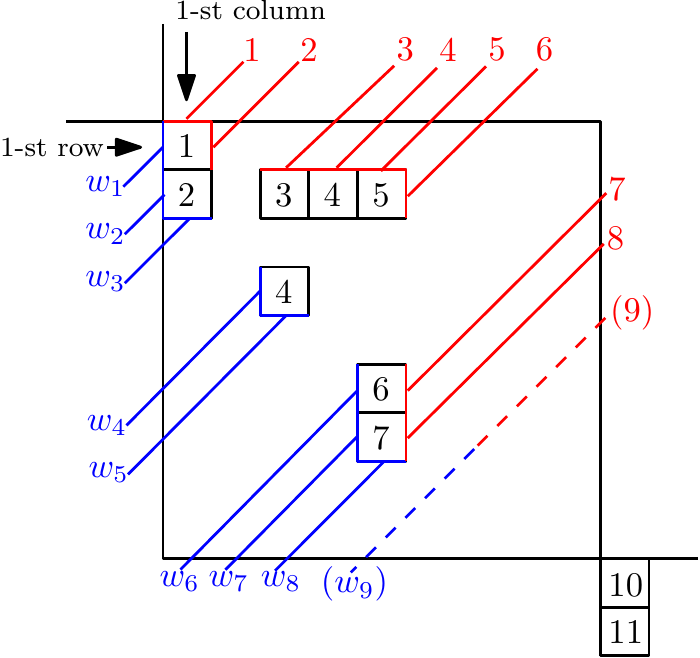}
\caption{content, (NE/SW-) boundaries, and rays}
\label{fig:boundary}
\end{figure}

\begin{lemma}
No two NE-rays (respectively, SW-rays) have the same index, and the indices increase as we read the rays from NW to SE direction.
\end{lemma}

\begin{proof}
If two NE-rays have the same index $i$, then it must be the case in which one ray is an E-boundary of a box $b_1$ with content $i-1$ and the other ray is an N-boundary of a box $b_2$ with content $i$. Our claim is that two boxes $b_1$ and $b_2$ should be in the same row or in the same column.

If one of the box is strictly to the southeast of the other, than it contradicts (C2). Thus one of the box should be strictly to the northeast of the other. If $b_1$ is to the northeast of $b_2$, then there must be a box $b_3$ above $b_2$ in the same row of $b_1$ by the NW condition and this contradicts that $b_2$ has N-boundary. On the other hand, if $b_2$ is to the northeast of $b_1$, then there is a box $b_3$ above $b_1$ in the same row of $b_2$ and the content of $b_3$ is less than $i-1$. This implies that there is a box with content $i-1$ between $b_3$ and $b_2$. This contradicts the fact that $b_1$ is the northeast-most box among all the boxes with content $i-1$.

We showed that $b_1$ and $b_2$ should be in the same row or in the same column. However, if they are in the same row then $b_1$ cannot have an E-boundary and if in the same column then $b_2$ cannot have an N-boundary. Hence, no two NE-rays can have the same index.

Finally, it is clear from (C1) and (C2) that the indices increase as we read the rays from NW to SE direction. The transposed version of the above argument will work for SW-rays.
\end{proof}
\medskip

\begin{lemma}~\label{lem:wirediag_fixedpoint}
There is no NE-ray of index $k$ if and only if there is no SW-ray of index $w_k$.
\end{lemma}
\begin{proof}
We will show that the followings are equivalent.
\begin{enumerate}[(1)]
\item
There is no N-boundary with content $k$ and no E-boundary of content $k-1$.
\item
There is no S-boundary with content $k-1$ and no W-boundary with content $k$
\item
There are no boxes with content $k$ or $k-1$.
\end{enumerate}

It is clear that (3) implies the other two. For (1)$\Rightarrow$(3), suppose there is at least one box with content $k$. Then, take the NE-most box $b$ with content $k$ and by the assumption there must be a box above $b$ with content $k-1$. Then, take the NE-most box $c$ with content $k-1$. By construction, this box $c$ cannot have a box to its right so the eastern edge of $c$ is an E-boundary, which is a contradiction. Similar argument shows that there are no box with content $k-1$.

The transposed version of the above argument shows (2)$\Rightarrow$(3).
\end{proof}
\medskip

\begin{figure}
\centering
\includegraphics{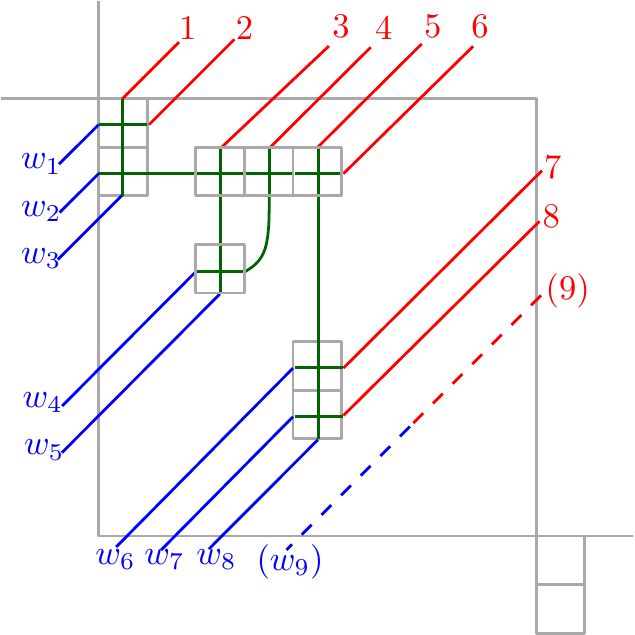}
\caption{wiring diagram}
\label{fig:wirediag}
\end{figure}

Now, given a NW affine diagram $D$ with a content map, we construct the \newword{wiring diagram} of $D$ through the following procedure.

\begin{enumerate}[(a)]
\item (Rays) Draw NE- and SW-rays.

\item (The ``Crosses'') Draw a ``$+$'' sign inside each box, i.e., connect the midpoint of the western edge to the midpoint of the eastern edge, and the midpoint of the northern edge to the midpoint of the southern edge of each box.

\item (Horizontal Movement) If the box $a$ and the box $b$ are in the same row ($a$ is to the left of $b$) and there are no boxes between them, then connect the midpoint of the eastern edge of $a$ to the midpoint of the western edge of $b$.

\item (Vertical Movement) If the box $a$ and the box $b$ are in the same column ($a$ is above $b$) and there are no boxes between them, then connect the midpoint of the southern edge of $a$ to the midpoint of the northern edge of $b$.

\item (The ``Tunnels'') Suppose that the box $a$ of content $k$ is not the northeast-most box among all the boxes with content $k$ and that there is no box on the same row to the right of $a$. Let $b$ be the closest box to $a$ such that it is to the northeast of $b$ and has content $k$. For every such pair $a$ and $b$, connect the midpoint of the eastern edge of $a$ to the midpoint of the southern edge of $b$.
\end{enumerate}

\begin{lemma}~\label{lem:wirediag_midpoint}
Each midpoint of an edge of a box in $D$ is connected to exactly two line segments of \textnormal{(a)}, \textnormal{(b)}, \textnormal{(c)}, \textnormal{(d)}, and \textnormal{(e)}. 
\end{lemma}
\begin{proof}
Note that NE- and SW-rays are drawn only when the horizontal/vertical movement is impossible at that midpoint. After one draws rays, crosses, horizontal/vertical lines, the remaining midpoints are connected by tunnels. 
\end{proof}

Figure \ref{fig:wirediag} illustrates the wiring diagram of the affine diagram of period $9$ in Figure \ref{fig:content}. Note that the curved line connecting two boxes of content $4$ is a ``tunnel''. Once we draw this wiring diagram of a NW affine diagram with a content, it is very easy to recover the affine permutation corresponding to the diagram. From a NE-ray indexed by $i$, proceed to the southwest direction following the lines in the wiring diagram until we meet a SW-ray of index $w_j$. This translates to $w_j = i$ in the corresponding affine permutation. If there is no NE-ray of index $i$ (equivalently, no SW-ray of index $w_i$), then let $w_i = i$.  For instance, Figure \ref{fig:wirediag} corresponds to the affine permutation $w = [w_1, w_2, \ldots, w_9] = [2,6,1,4,3,7,8,5,9] \in \Sigma_9 \subset \widetilde{\Sigma}_9$

\begin{proposition}~\label{prop:wirediag}
The wiring diagram gives a bijection between the NW affine diagrams of period $n$ with a content map, and the affine permutations in $\widetilde{\Sigma}_n$.
\end{proposition}
\begin{proof}
Let $D$ be an NW affine diagram of period $n$ with a content map and suppose we drew a wiring diagram on $D$ by the above rules. For every $k$ not appearing in the indices of NE-rays, draw a ``fixed point'' ray from northeast to southwest using Lemma \ref{lem:wirediag_fixedpoint} with NE index $k$ and SW index $w_k$ (see Figure \ref{fig:wirediag}, $w_9 =9$.) Now the indices of the NE- and SW-rays will cover all the integers, and there is a one-to-one correspondence between indices of NE-rays and SW-rays following the wires (Lemma \ref{lem:wirediag_midpoint}). Let $f(a) = b$ if the NE-ray $b$ corresponds to the SW-ray $w_a$ following the wires. We will show that $w = (f(i))_{i\in \ZZ}$ is the affine permutation corresponding to $D$.

Consider two wires corresponding to SW-rays $w_i$ and $w_j$, $i < j$. It is easy to see that two wires intersect at most once, and the crosses inside the boxes exactly correspond to these intersections. This implies the two wires intersect if and only if $(i,j)$ is an inversion, and each box corresponds to these inversions. Moreover, the SW-ray $w_i$ must enter into a W-boundary of a box with content $i$ and the NE-ray $f(j)$ should come out from a N-boundary of a box with content $f(j)$. Hence the intersection should occur in the box with coordinate $(i, f(j))$. This concludes that the diagram $D$ is indeed a diagram of an affine permutation $w$.
\end{proof}

Our main result of this section, Theorem~\ref{thm:diagchar}, is a direct consequence of Proposition~\ref{prop:wirediag}.

\bigskip

\begin{acknowledgement}
\label{sec:ack}
We thank Sara Billey, Thomas Lam, and Richard Stanley for helpful discussions. 
\end{acknowledgement}

\bibliographystyle{plain}
\bibliography{thebibliography}
\label{sec:biblio}

\end{document}